\documentclass[twoside,11pt]{article}
\usepackage[preprint]{jmlr2e}
\usepackage{amsmath,mathtools}
\usepackage{enumitem,microtype,booktabs,tabularx,array,multirow,xcolor,needspace}
\hypersetup{
pdftitle={Lower Bounds for Stochastic First-Order Algorithms with Variance Reduction in Nonconvex--Concave Minimax Optimization}, pdfauthor={Jiayi Song and Zi Xu}}
\setlist{nosep}
\allowdisplaybreaks[2]
\numberwithin{equation}{section}
\numberwithin{theorem}{section}

\ShortHeadings{Lower Bounds for SFO Algorithms with VR in Minimax Optimization}{Song and Xu}
\firstpageno{1}

\title{Lower Bounds for Stochastic First-Order Algorithms with Variance
Reduction in Nonconvex--Concave Minimax Optimization}
\author{\name Jiayi Song \email songjy715@shu.edu.cn \\
\name Zi Xu\thanks{Corresponding author.} \email xuzi@shu.edu.cn \\
\addr Department of Mathematics, College of Sciences \\
Shanghai University, Shanghai 200444, China}

\renewenvironment{proof}[1][Proof]
{\par\noindent\textbf{#1.}\enspace\ignorespaces} {\unskip\nobreak\hfill\BlackBox\par\medskip}

\begin{document}

\maketitle

\begin{abstract}
We establish complexity lower bounds for stochastic first-order
methods in nonconvex--concave minimax optimization that remain
valid when variance reduction is allowed.
Our main contribution is a nonconvex--concave lower bound for
an algorithm class that accommodates variance reduction,
rather than restricting algorithms to exclude it.
We consider zero-respecting randomized algorithms with arbitrary
finite nonadaptive batches of same-sample gradient queries,
an access model that accommodates variance-reduction methods.
Our construction limits coordinate revelation across oracle
rounds despite this stronger access.
For objectives with an $L$-Lipschitz joint gradient, a compact
convex dual domain of Euclidean radius at most $D_Y$, and an
initial primal value-function gap at most $\Delta$, we measure
stationarity by the gradient norm of the Moreau envelope of the
constrained primal value function with parameter $1/(2L)$.
Under an unbiased stochastic oracle with variance at most
$\sigma^2$ and mean-square smoothness, we prove the oracle-round
lower bound
$\Omega\!\left(
L^2D_Y\Delta\varepsilon^{-3}
+
L^3D_Y^2\Delta\sigma^2\varepsilon^{-6}
\right)$
in the stated parameter regime.
The bound therefore captures complexity barriers that persist
even for algorithms using variance reduction.
We also establish complementary lower bounds for
nonconvex--strongly-concave minimax optimization.
With dual strong-concavity parameter $\mu$ and condition number
$\kappa:=L/\mu$, combining our stochastic constructions with a
deterministic baseline for the same algorithm class yields
$\Omega\!\left(
L\Delta\sqrt{\kappa}\,\varepsilon^{-2}
+
L\Delta\kappa\sigma^2\varepsilon^{-4}
\right)$
under bounded variance, and
$\Omega\!\left(
L\Delta\sqrt{\kappa}\,\varepsilon^{-2}
+
\Delta\bar L\sigma\kappa^{3/2}\varepsilon^{-3}
\right)$
under additional mean-square smoothness with constant $\bar L$,
in their respective parameter regimes.
Together, these results quantify accuracy-, noise-, and
conditioning-dependent barriers across the concave and
strongly concave regimes.
\end{abstract}

\begin{keywords}nonconvex--concave minimax optimization, stochastic first-order lower bounds, mean-squared smoothness, variance reduction, zero-respecting algorithms
\end{keywords}

\section{Introduction}\label{sec:introduction}
Minimax optimization provides a natural framework for robust learning, adversarial training, and related problems in which a decision is optimized against a worst-case response~\cite{lin-gda,sreda}.
We study the stochastic first-order oracle complexity of problems of the form
\begin{equation*}
\min_{x\in X}\Phi(x),
\qquad
\Phi(x):=\max_{y\in Y}F(x,y),
\end{equation*}
where $X\subseteq\mathbb{R}^n$ is closed and convex,
$Y\subseteq\mathbb{R}^p$ is compact and convex, and $F$ is jointly smooth.
Our main focus is the nonconvex--concave (NC--C) setting, in which $F$ may be nonconvex in $x$ and is concave in $y$.
We also consider the nonconvex--strongly-concave (NC--SC) setting as a complementary regime.
Strong concavity guarantees a unique dual maximizer and a differentiable primal value function, whereas the NC--C value function may be nonsmooth.
We therefore measure stationarity through the Moreau envelope of the constrained primal value function.

Existing research has established first-order upper and lower bounds for both minimax regimes.
In particular, deterministic NC--C optimization has an $\mathcal{O}(\varepsilon^{-3})$ complexity dependence, while stochastic methods attain an $\mathcal{O}(\varepsilon^{-6})$ dependence under bounded variance and mean-squared smoothness~\cite{wu-ncc,pan-ncc,lin-ppa,spde}.
However, some existing stochastic lower bounds restrict oracle access in ways that exclude variance reduction.
They therefore do not directly resolve whether the same complexity barrier persists for algorithms allowed to use coupled gradient evaluations with a shared random sample.
This distinction matters because variance reduction improves the stochastic accuracy dependence in single-level nonconvex optimization~\cite{arjevani,spider}.

Our central goal is to establish lower bounds for stochastic first-order NC--C minimax optimization \emph{when variance reduction is allowed}.
We construct a mean-square-smooth stochastic oracle and prove a lower bound for zero-respecting algorithms that permits the same-sample queries used by variance-reduction methods.
The result retains the
$\mathcal{O}(\varepsilon^{-6})$ stochastic dependence, extending the scope of
existing lower bounds rather than introducing a new accuracy exponent.
We also develop complementary NC--SC lower bounds under bounded variance and mean-squared smoothness, using a common chain-based construction.

\subsection{Related work}\label{sec:related-work}

\paragraph{Nonconvex optimization.}
\cite{carmon,carmon-ii} developed zero-chain constructions
and established the deterministic first-order lower bound
$\Omega(L\Delta\varepsilon^{-2})$ under an $L$-Lipschitz gradient and an
initial function-value gap bounded by $\Delta$.
For stochastic first-order oracles, Arjevani et al.~\cite{arjevani} established
$\Omega(L\Delta\varepsilon^{-2}
+L\Delta\sigma^2\varepsilon^{-4})$ under bounded variance, where $\sigma^2$ bounds the oracle variance.
Under mean-squared smoothness with constant $\bar L$, they obtained
$\Omega(\Delta\bar L\sigma\varepsilon^{-3}
+(\Delta\bar L+\sigma^2)\varepsilon^{-2})$.
This matches the accuracy dependence of recursive variance-reduction methods such as SPIDER~\cite{spider}.
Their oracle model also permits the same-sample gradient evaluations considered here, but their single-level hard instances do not capture primal--dual interactions.

\paragraph{Nonconvex--concave minimax optimization.}
Under Moreau-envelope stationarity, \cite{wu-ncc} established the deterministic lower bound
$\Omega(L^2D_Y\Delta\varepsilon^{-3})$
and the bounded-variance stochastic lower bound
$\Omega(L^3D_Y^2\Delta\sigma^2\varepsilon^{-6})$,
where $D_Y$ bounds the dual-domain radius.
Pan et al.~\cite{pan-ncc} obtained the same deterministic
$\varepsilon^{-3}$ order under their normalization.
Pan and Li~\cite{pan-pl} considered a different inner structure satisfying a Polyak--\L{}ojasiewicz condition.

On the upper-bound side, deterministic first-order methods achieve essentially
$\widetilde O(L^2D_Y\Delta\varepsilon^{-3})$
complexity~\cite{lin-ppa}.
In the stochastic setting, SAPD+ and SPDE achieve $O(\varepsilon^{-6})$ oracle complexity under bounded variance, while VR-SPDE achieves $O(\varepsilon^{-6})$ under mean-squared smoothness~\cite{spde}.
These results motivate a lower bound that explicitly allows variance reduction.
Our main result supplies such a bound: it preserves the
$\mathcal{O}(\varepsilon^{-6})$ stochastic barrier under mean-squared smoothness
for a zero-respecting algorithm class with same-sample oracle access.

\paragraph{Nonconvex--strongly-concave minimax optimization.}
Let $\kappa:=L/\mu$, where $\mu$ is the dual strong-concavity parameter.
\cite{li} established the deterministic lower bound
$\Omega(L\Delta\sqrt{\kappa}\,\varepsilon^{-2})$
for zero-respecting first-order algorithms and the bounded-variance stochastic lower bound
$\Omega(L\Delta\sqrt{\kappa}\,\varepsilon^{-2}
+L\Delta\kappa^{1/3}\sigma^2\varepsilon^{-4})$.
\cite{zhang-ncsc} studied related deterministic and
finite-sum bounds using a different construction and oracle formulation.
More recently, Zhang et al.~\cite{wu-ncsc} established
$\Omega(L\Delta\kappa\sigma^2\varepsilon^{-4})$
under bounded variance and
$\Omega(\bar L\Delta\sigma\bar\kappa^{3/2}\varepsilon^{-3})$
under mean-squared smoothness, with $\bar\kappa$ denoting the condition parameter used in their formulation.
Their hard instances use a quadratic lifting of a single-level nonconvex zero-chain.

On the upper-bound side, Lin et al.~\cite{lin-ppa} established
$\widetilde O(L\Delta\sqrt{\kappa}\,\varepsilon^{-2})$
deterministic complexity.
SPDE and VR-SPDE achieve
$O(\kappa\varepsilon^{-4})$ and
$O(\kappa^{3/2}\varepsilon^{-3})$ stochastic oracle complexity,
respectively~\cite{spde}.
Our NC--SC results complement existing lower bounds through a product-weighted chain that also underlies our main NC--C construction.
Under the same stationarity criterion and stochastic oracle conventions, the stochastic terms have the same $\kappa$ and $\varepsilon$ dependences as the corresponding upper bounds.

\subsection{Main results and contributions}\label{sec:contributions}

Our main contribution is an NC--C lower bound for stochastic first-order algorithms that are allowed to use variance reduction.
The NC--SC results provide complementary bounds and demonstrate how a common construction applies to both minimax regimes.

\begin{enumerate}[leftmargin=*,label=(\roman*),itemsep=5pt]
\item \textbf{An NC--C lower bound allowing variance reduction.}
We establish
$\Omega(L^2D_Y\Delta\varepsilon^{-3}
+L^3D_Y^2\Delta\sigma^2\varepsilon^{-6})$ for zero-respecting algorithms under an unbiased, bounded-variance, mean-square-smooth oracle; see Theorem~\ref{nccmss:thm:zr}.
The algorithm class permits same-sample queries and therefore allows variance reduction.
We realize a smooth nonconvex zero-chain as the value function of a concave minimax problem with a bounded dual domain.
A dual path delays coordinate revelation, while a Bernoulli gate preserves this delay and the value function under mean-squared smoothness.
The resulting stochastic lower bound has the same
$\mathcal{O}(\varepsilon^{-6})$ dependence as VR-SPDE~\cite{spde}.

\item \textbf{A complementary bounded-variance NC--SC lower bound.}
We establish
$\Omega(L\Delta\sqrt{\kappa}\,\varepsilon^{-2}
+L\Delta\kappa\sigma^2\varepsilon^{-4})$; see Theorem~\ref{thm:intro-bv}.
A quadratic strongly-concave dual lift produces the deterministic propagation term, and a clipped stochastic dual potential hides the next block with positive probability, yielding the stochastic term.

\item \textbf{A complementary mean-square-smooth NC--SC lower bound.}
We establish
$\Omega(L\Delta\sqrt{\kappa}\,\varepsilon^{-2}
+\Delta\bar L\sigma\kappa^{3/2}\varepsilon^{-3})$; see Theorem~\ref{thm:mss-sum}.
We retain the delayed strongly-concave chain and replace the clipped noise potential by a smooth stochastic potential whose sample gradients satisfy both bounded variance and mean-squared smoothness.
These NC--SC results complement recent work rather than claim priority for this regime.
\end{enumerate}

Table~\ref{tab:intro-comparison} summarizes selected stochastic first-order upper and lower bounds.

\begin{table}[t]
\centering
\caption{Selected stochastic first-order complexity bounds for NC--SC and
NC--C minimax optimization.
Only the dependence on $\kappa$ and
$\varepsilon$ is displayed.
MSS denotes mean-squared smoothness,
and ZR denotes zero-respecting algorithms.}
\label{tab:intro-comparison}
\setlength{\tabcolsep}{4pt}
\renewcommand{\arraystretch}{1.18}
\small
\resizebox{\textwidth}{!}{%
\begin{tabular}{@{}l c c c l l@{}}
\toprule
Problem class
& MSS
& Algorithm class
& Result
& Reference
& Oracle complexity \\
\midrule

\multirow{7}{*}{NC--SC}
& No
& General
& Upper bound
& SPDE~\cite{spde}
& $O(\kappa\varepsilon^{-4})$ \\

& No
& ZR
& Lower bound
& Li et al.~\cite{li}
& $\Omega(\kappa^{1/3}\varepsilon^{-4})$ \\

& No
& ZR
& Lower bound
& Zhang et al.~\cite{wu-ncsc}
& $\Omega(\kappa\varepsilon^{-4})$ \\

& \textbf{No}
& \textbf{ZR}
& \textbf{Lower bound}
& \textbf{Our result}
& $\boldsymbol{\Omega(\kappa\varepsilon^{-4})}$ \\

\cmidrule(lr){2-6}

& Yes
& General
& Upper bound
& VR-SPDE~\cite{spde}
& $O(\kappa^{3/2}\varepsilon^{-3})$ \\

& Yes
& ZR
& Lower bound
& Zhang et al.~\cite{wu-ncsc}
& $\Omega(\kappa^{3/2}\varepsilon^{-3})$ \\

& \textbf{Yes}
& \textbf{ZR}
& \textbf{Lower bound}
& \textbf{Our result}
& $\boldsymbol{\Omega(\kappa^{3/2}\varepsilon^{-3})}$ \\

\midrule

\multirow{2}{*}{NC--C}
& Yes
& General
& Upper bound
& VR-SPDE~\cite{spde}
& $O(\varepsilon^{-6})$ \\

& \textbf{Yes}
& \textbf{ZR}
& \textbf{Lower bound}
& \textbf{Our result}
& $\boldsymbol{\Omega(\varepsilon^{-6})}$ \\

\bottomrule
\end{tabular}%
}
\end{table}

\paragraph{Notation, definitions, and assumptions.}
The norm $\|\cdot\|$ and inner product $\langle\cdot,\cdot\rangle$ are Euclidean.
For a differentiable vector-valued map $G$, $\nabla G$ denotes its Jacobian, so $\nabla G(u)^\top v$ is the gradient of $\langle G(u),v\rangle$.
A function is $C^{1,1}$ if it is continuously differentiable with Lipschitz gradient.
The notation $O$ and $\Omega$ suppresses constant factors, while $\widetilde O$ additionally suppresses logarithmic factors.

For our problem classes, $L$ is the joint gradient Lipschitz constant, $D_Y$ bounds the dual-domain radius, and $\Delta$ bounds the initial value-function gap: $\max_{y\in Y}\|y\|\le D_Y$ and $\Phi(0)-\inf_{x\in X}\Phi(x)\le\Delta$, with $0\in X$.
In the NC--SC setting, $\mu>0$ is the dual strong-concavity parameter and $\kappa=L/\mu$.
Let $\phi:=\Phi+\iota_X$, where $\iota_X$ is zero on $X$ and $+\infty$ outside $X$.
For the NC--C class, we use $S_{\rm OS}(x):=\|\nabla\phi_{1/(2L)}(x)\|$, where $\phi_{1/(2L)}$ is the Moreau envelope of $\phi$.
A random primal output $\widehat x$ is $\varepsilon$-stationary in expectation when $\mathbb E S_{\rm OS}(\widehat x)\le\varepsilon$.
For the NC--SC class, we instead use $\mathbb E\|\nabla\Phi(\widehat x)\|\le\varepsilon$.

For a joint query point $z=(x,y)$ and random sample $\omega$, the oracle returns $G(z;\omega)$ satisfying $\mathbb E_\omega[G(z;\omega)]=\nabla F(z)$ and $\mathbb E_\omega\|G(z;\omega)-\nabla F(z)\|^2\le\sigma^2$.
Under mean-squared smoothness, it also satisfies $\mathbb E_\omega\|G(z;\omega)-G(z';\omega)\|^2 \le\bar L^2\|z-z'\|^2$, with the same sample used at both points.
Zero-respecting algorithms cannot activate coordinates before they are revealed by oracle responses, but our mean-squared smoothness model permits same-sample queries.
Formal definitions and parameter conditions are given in Section~\ref{common:models} and the corresponding theorems.
\paragraph{Organization.}\label{sec:organization}
Section~\ref{common:models} defines the function classes, stationarity criteria, and oracle models.
Section~\ref{sec:common-chain} develops the common smooth chain.
Section~\ref{nccmss:sec:zr} proves the NC--C result.
Section~\ref{sec:ncsc} develops the NC--SC constructions and establishes the lower bounds under bounded variance and mean-squared smoothness.
Section~\ref{sec:conclusion} concludes the paper.

\section{Problem setup and oracle complexity}
\label{common:models}

We specify the problem classes, stationarity criterion, and stochastic
oracle access used throughout the paper. We then define the
zero-respecting algorithm class and distinguish oracle rounds from
individual gradient evaluations.

\subsection{Problem classes and stationarity}

Let $X\subseteq\mathbb R^n$ and $Y\subseteq\mathbb R^p$ be nonempty
closed convex sets, with $0\in X$ and $Y$ bounded. For a continuously
differentiable objective $F:X\times Y\to\mathbb R$, define
$\Phi(x):=\max_{y\in Y}F(x,y)$ and
$\Phi_*:=\inf_{x\in X}\Phi(x)$.
The maximum is attained because $Y$ is compact.
Differentiability on a closed domain is understood through a
continuously differentiable extension to a neighborhood of that domain.

\begin{definition}[Nonconvex--concave function class]
	\label{def:unified-class}
	For $L,\Delta,D_Y>0$, let
	$\mathcal F_{\rm NC\text{-}C}(L,\Delta,D_Y)$
	denote the class of such objectives, over arbitrary finite dimensions,
	satisfying the following conditions:
	\begin{enumerate}[label=(\roman*),leftmargin=*]
		\item \textbf{Joint smoothness.}
		The joint gradient is $L$-Lipschitz:
		\begin{equation}
			\label{eq:class-smoothness}
			\|\nabla F(z)-\nabla F(z')\|
			\le L\|z-z'\|,
			\qquad z,z'\in X\times Y.
		\end{equation}
		
		\item \textbf{Dual concavity.}
		For every $x\in X$, the function $F(x,\cdot)$ is concave on $Y$.
		
		\item \textbf{Bounded dual domain.}
		The dual radius satisfies $\max_{y\in Y}\|y\|\le D_Y$.
		
		\item \textbf{Bounded initial gap.}
		The value function satisfies
		\begin{equation}
			\label{eq:class-gap}
			\Phi(0)-\Phi_*\le\Delta.
		\end{equation}
	\end{enumerate}
	No convexity is required in the primal variable.
\end{definition}

In the NC--SC setting, dual concavity is strengthened to
$\mu$-strong concavity for some $\mu>0$, and
$\kappa:=L/\mu$ denotes the condition number.
Any additional domain or parameter restrictions are specified in the
corresponding lower-bound statements.

\begin{definition}[Optimization stationarity]
	\label{def:optimization-stationarity}
	Let $\phi:=\Phi+\iota_X$, where $\iota_X$ is zero on $X$ and
	$+\infty$ outside $X$. Its Moreau envelope at parameter $1/(2L)$ is
	$\phi_{1/(2L)}(x):=
	\min_{u\in X}\{\Phi(u)+L\|u-x\|^2\}$,
	for $x\in\mathbb R^n$.
	Since $\phi$ is proper, closed, and $L$-weakly convex, the minimizer
	$x^+(x)$ is unique, and
	$\nabla\phi_{1/(2L)}(x)=2L(x-x^+(x))$.
	
	Define
	$S_{\rm OS}(x):=\|\nabla\phi_{1/(2L)}(x)\|$.
	A random output $\widehat x\in X$ is
	$\varepsilon$-stationary if
	$\mathbb E[S_{\rm OS}(\widehat x)^2]\le\varepsilon^2$,
	where the expectation includes both oracle and algorithmic randomness.
\end{definition}

This criterion applies to both NC--C and NC--SC problems.
It also implies
$\mathbb E[S_{\rm OS}(\widehat x)]\le\varepsilon$
by Jensen's inequality.

\subsection{Stochastic oracle models}

\begin{definition}[Stochastic first-order oracle]
	\label{def:stochastic-oracle}
	For $\sigma\ge0$, a bounded-variance stochastic first-order oracle
	returns
	$\mathcal O_F^{\mathrm{sfo}}(z;\xi)
	=(F(z),\widehat\nabla F(z;\xi))$
	at a joint query point $z=(x,y)\in X\times Y$.
	The sample gradient is jointly measurable and satisfies
	\begin{equation*}
		\mathbb E_\xi[\widehat\nabla F(z;\xi)]=\nabla F(z),
		\qquad
		\mathbb E_\xi
		\|\widehat\nabla F(z;\xi)-\nabla F(z)\|^2
		\le\sigma^2,
		\qquad z\in X\times Y.
	\end{equation*}
	Thus, the oracle provides an exact function value and an unbiased
	stochastic joint gradient.
	
	A mean-square-smooth (MSS) oracle additionally satisfies
	\begin{equation*}
		\mathbb E_\xi
		\|\widehat\nabla F(z;\xi)-\widehat\nabla F(z';\xi)\|^2
		\le\bar L^2\|z-z'\|^2,
		\qquad z,z'\in X\times Y,
	\end{equation*}
	where the same sample $\xi$ is used at both points.
\end{definition}

The MSS condition concerns a coupled pair of sample gradients.
To make this coupling available to an algorithm, we explicitly allow
the same-sample batch access defined below.

\subsection{Algorithm class and complexity}

\begin{definition}[Zero-respecting algorithms with same-sample access]
	\label{def:zero-respecting-algorithm}
	Consider an instance with $(0,0)\in X\times Y$.
	At round $t\ge1$, an algorithm chooses a finite batch
	$z^{t,1},\ldots,z^{t,K_t}\in X\times Y$, with $K_t\ge1$,
	using only its internal randomness and responses from earlier rounds.
	The entire batch is chosen before any response from round $t$ is
	observed.
	
	The oracle then draws a fresh sample $\xi_t$, independently of the
	past, and returns
	$(F(z^{t,k}),g^{t,k})$ for each $k\in[K_t]$, where
	$g^{t,k}:=\widehat\nabla F(z^{t,k};\xi_t)$.
	All query points within the batch share the same sample.
	
	Let $\mathcal S_0:=\varnothing$, and define the revealed joint
	coordinates after round $t$ by
	\begin{equation*}
		\mathcal S_t:=
		\bigcup_{s=1}^{t}\ \bigcup_{k=1}^{K_s}
		\operatorname{supp}(g^{s,k}).
	\end{equation*}
	Let $\mathcal S_t^x:=\mathcal S_t\cap[n]$ denote the revealed primal
	coordinates, with primal coordinates indexed first in the joint vector.
	
	The algorithm is zero-respecting if, almost surely,
	$\operatorname{supp}(z^{t,k})\subseteq\mathcal S_{t-1}$
	for every query, and its primal output $\widehat x^{\,t}$ after
	$t$ rounds satisfies
	$\operatorname{supp}(\widehat x^{\,t})\subseteq\mathcal S_t^x$.
\end{definition}

The origin condition ensures that the first zero-respecting query is
feasible; it holds for the hard instances considered here.
Function values may influence the choice of subsequent queries, but
do not enlarge the revealed-coordinate set.

One batch counts as one \emph{oracle round}; its $K_t$ sample-gradient
evaluations count as $K_t$ individual oracle evaluations.
Thus, after $T$ rounds the number of gradient evaluations is
$\sum_{t=1}^{T}K_t\ge T$.
We prove lower bounds on the number of rounds in this batch-access
model, which also imply lower bounds on the number of individual
gradient evaluations.

An $\varepsilon$-stationarity complexity lower bound means that every
algorithm in the specified class, with a smaller prescribed budget,
fails to satisfy
$\mathbb E[S_{\rm OS}(\widehat x)^2]\le\varepsilon^2$
on some admissible objective--oracle pair.
The same-sample access permits coupled gradient differences used by
variance-reduction methods, subject to the zero-respecting restriction
and the nonadaptive choice of queries within each round.

\section{Common hard-instance framework}
\label{sec:common-chain}
\label{sec:deterministic-baselines}

Our constructions share a product-weighted nonconvex zero-chain.
We first state its analytic and coordinate-revelation properties.
We then transfer its gradient barrier to Moreau-envelope stationarity
and give a probabilistic bound on coordinate revelation.
The minimax lifts and stochastic oracle constructions are developed
in the subsequent sections.

\subsection{A product-weighted zero-chain}

For a scalar $u\in\mathbb R$, define the activation function
\begin{equation*}
	S(u):=
	\begin{cases}
		0, & u\le\frac14,\\
		\sin^2\!\left(\frac{2\pi}{3}(u-\frac14)\right),
		& \frac14<u<1,\\
		1, & u\ge1,
	\end{cases}
\end{equation*}
and the bounded profile
$\mathcal K(u):=\frac3{10}(1+\gamma(u)^2)-\frac25\gamma(u)$,
where
\begin{equation*}
	\gamma(u):=
	\begin{cases}
		\frac{u}{4}, & |u|\le1,\\
		\dfrac{\operatorname{sgn}(u)}4
		\left(2-\dfrac{(3-|u|)^2}{4}\right),
		& 1<|u|<3,\\
		\frac12\operatorname{sgn}(u), & |u|\ge3.
	\end{cases}
\end{equation*}
The negative-tail penalty is
\begin{equation*}
	J(u):=
	\begin{cases}
		0, & u\ge-\frac12,\\
		\frac12(u+\frac12)^2, & -\frac32<u<-\frac12,\\
		-u-1, & u\le-\frac32.
	\end{cases}
\end{equation*}

For $x\in\mathbb R^m$, let $P_0(x):=1$,
$P_i(x):=\prod_{j=1}^{i}S(x_j)$, and
$W_i(x):=P_{i-1}(x)(1-S(x_i))$ for $i\in[m]$.
Define
\begin{equation*}
	\varphi_m(x):=
	-4\sum_{i=1}^{m}S(x_i)
	+\sum_{i=1}^{m}J(x_i)
	+\sum_{i=1}^{m}W_i(x)^2\mathcal K(x_i).
\end{equation*}
For later use, write
$H_m(x):=-4\sum_{i=1}^{m}S(x_i)+\sum_{i=1}^{m}J(x_i)$,
$G_i(x):=W_i(x)\sqrt{\mathcal K(x_i)}$,
$G(x):=(G_1(x),\ldots,G_m(x))$, and
$Q_m(x):=\|G(x)\|^2$.
Then $\varphi_m=H_m+Q_m$.
The square root is well defined because
$\gamma(u)\in[-1/2,1/2]$ and $\mathcal K(u)\ge7/40$.

The product in $W_i$ vanishes whenever an earlier coordinate is at
most $1/4$. Together with the profile $\mathcal K$ and the
negative-tail penalty $J$, this yields the following properties.

\begin{lemma}[Properties of the common chain]
	\label{lem:common-chain}
	There exist universal constants $C_\varphi,c_b>0$ such that, for every
	$m\ge1$, the function $\varphi_m:\mathbb R^m\to\mathbb R$ satisfies:
	\begin{enumerate}[label=(\roman*),leftmargin=*]
		\item \textbf{Dimension-independent smoothness.}
		For all $x,x'\in\mathbb R^m$,
		$\|\nabla\varphi_m(x)-\nabla\varphi_m(x')\|
		\le C_\varphi\|x-x'\|$.
		
		\item \textbf{Initial gap.}
		$\varphi_m(0)-\inf_{x\in\mathbb R^m}\varphi_m(x)\le5m$.
		
		\item \textbf{Gradient barrier.}
		If $x_m\le1/2$, then $\|\nabla\varphi_m(x)\|\ge c_b$.
		
		\item \textbf{Zero-chain property.}
		For every $r\in\{0,\ldots,m-1\}$,
		$\operatorname{supp}(x)\subseteq[r]$ implies
		$\operatorname{supp}(\nabla\varphi_m(x))\subseteq[r+1]$,
		where $[0]:=\varnothing$.
	\end{enumerate}
\end{lemma}

The proof is given in Appendix~\ref{app:common}.
Properties (i)--(ii) control the scaling and length of the hard
instances. Property (iii) supplies a gradient barrier, while
property (iv) restricts coordinate revelation.
The minimax constructions below introduce additional delays through
dual variables and stochastic gates; those delays require separate
verification for each oracle construction.

\subsection{From a gradient barrier to optimization stationarity}

The next lemma transfers the chain's gradient barrier to the
stationarity criterion used in our lower bounds.
It concerns an unconstrained primal domain.

\begin{lemma}[Moreau-envelope barrier]
	\label{nccmss:lem:moreau-transfer}
	Let $m\ge1$, $\alpha,\beta>0$, and
	$\Phi(x):=\alpha\varphi_m(x/\beta)$ on $X=\mathbb R^m$.
	Set $\phi:=\Phi$, and let $c_b$ be the constant in
	Lemma~\ref{lem:common-chain}.
	If $\Phi$ is $L$-smooth and
	$0<b\le\min\{L\beta,\alpha c_b/(2\beta)\}$,
	then $S_{\rm OS}(x)>b$ for every $x\in\mathbb R^m$ with $x_m=0$.
\end{lemma}

\begin{proof}
	Fix $x$ with $x_m=0$, and let $x^+$ be the unique minimizer of
	$u\mapsto\Phi(u)+L\|u-x\|^2$.
	The unconstrained first-order condition gives
	\begin{equation*}
		\nabla\phi_{1/(2L)}(x)
		=2L(x-x^+)
		=\nabla\Phi(x^+)
		=\frac{\alpha}{\beta}
		\nabla\varphi_m(x^+/\beta).
	\end{equation*}
	Suppose that $S_{\rm OS}(x)\le b$. Then
	\begin{equation*}
		\frac{|x_m^+|}{\beta}
		\le\frac{\|x^+-x\|}{\beta}
		=\frac{S_{\rm OS}(x)}{2L\beta}
		\le\frac{b}{2L\beta}
		\le\frac12.
	\end{equation*}
	Lemma~\ref{lem:common-chain}(iii) therefore yields
	$S_{\rm OS}(x)\ge\alpha c_b/\beta\ge2b>b$,
	a contradiction.
\end{proof}

\subsection{Stochastic coordinate revelation}

The remaining ingredient converts a required number of successful
oracle rounds into a lower bound on stationarity.
The coordinate-revelation premise in the following lemma is verified
separately for each hard instance.

\begin{lemma}[Bernoulli progress bound]
	\label{lem:bernoulli-progress}
	Let $M\ge1$ be an integer and $p\in[0,1]$.
	At each oracle round $t$, draw a fresh
	$\omega_t\sim\operatorname{Bernoulli}(p)$, independently of the past,
	and share it across all queries in that round.
	Call a round successful if $\omega_t=1$.
	
	Suppose that, almost surely, revealing the last primal coordinate
	requires at least $M$ successful rounds.
	For any zero-respecting algorithm using at most $N$ rounds, its output
	$\widehat x$ satisfies:
	\begin{enumerate}[label=(\roman*),leftmargin=*]
		\item The last coordinate remains zero with probability
		\begin{equation}
			\label{eq:bernoulli-progress}
			\Pr(\widehat x_m=0)\ge1-\frac{Np}{M}.
		\end{equation}
		
		\item If $S_{\rm OS}(x)\ge b>0$ for every feasible $x$ with
		$x_m=0$, then, whenever $Np\le M$,
		\begin{equation*}
			\mathbb E[S_{\rm OS}(\widehat x)]
			\ge b\left(1-\frac{Np}{M}\right),
			\qquad
			\mathbb E[S_{\rm OS}(\widehat x)^2]
			\ge b^2\left(1-\frac{Np}{M}\right).
		\end{equation*}
	\end{enumerate}
\end{lemma}

\begin{proof}
	If the algorithm stops before round $N$, extend the Bernoulli sequence
	with unused independent draws.
	By the revelation premise and the zero-respecting output restriction,
	$\widehat x_m\ne0$ implies
	$\sum_{t=1}^{N}\omega_t\ge M$.
	Markov's inequality gives
	\begin{equation*}
		\Pr(\widehat x_m\ne0)
		\le
		\Pr\left(\sum_{t=1}^{N}\omega_t\ge M\right)
		\le
		\frac{\mathbb E[\sum_{t=1}^{N}\omega_t]}{M}
		=\frac{Np}{M},
	\end{equation*}
	which proves (i).
	
	For $q\in\{1,2\}$, nonnegativity and the assumed stationarity barrier
	give
	\begin{equation*}
		\mathbb E[S_{\rm OS}(\widehat x)^q]
		\ge
		\mathbb E\!\left[
		S_{\rm OS}(\widehat x)^q
		\mathbf 1_{\{\widehat x_m=0\}}
		\right]
		\ge b^q\Pr(\widehat x_m=0)
		\ge b^q\left(1-\frac{Np}{M}\right).
	\end{equation*}
	This proves (ii).
\end{proof}

Consequently, each lower-bound construction has three tasks:
verify the prescribed objective and oracle conditions, establish a
stationarity barrier before the last primal coordinate is revealed,
and bound the number of successful rounds needed to reveal it.
The preceding lemmas separate these tasks from the construction-specific
primal--dual mechanisms.

\section{The stochastic NC--C lower bound}
\label{nccmss:sec:zr}

We construct an NC--C instance whose stochastic oracle delays primal
coordinate revelation even under same-sample batch access.
The construction preserves the value function of the common chain,
while a Bernoulli gate controls propagation within each dual block.
We first establish the stochastic lower bound and then combine it
with a deterministic NC--C lower bound.

\subsection{Objective and stochastic oracle}

Recall the common chain
$\varphi_m(u)=H_m(u)+\sum_{i=1}^{m}W_i(u)^2\mathcal K(u_i)$
from Section~\ref{sec:common-chain}.
Let $c_b$ be the constant in Lemma~\ref{lem:common-chain}(iii),
and let $C_*>0$ be the universal constant specified in
Appendix~\ref{app:ncc}.
Set
\begin{equation*}
	\zeta:=\frac{4}{D_Y},
	\qquad
	\beta:=\frac{24C_*\varepsilon}{c_bL},
	\qquad
	\alpha:=\frac{L\beta^2}{8C_*},
	\qquad
	m:=\left\lfloor\frac{\Delta}{10\alpha}\right\rfloor.
\end{equation*}
The parameter conditions below ensure that $m\ge1$.

Take $X:=\mathbb R^m$ and
$Y:=\{y\in\mathbb R^{2m}:\|y\|\le D_Y/2\}$.
Write $y=(y^1,\ldots,y^m)$ with $y^i\in\mathbb R^2$,
and define $u:=x/\beta$, $v:=\zeta y$, and
$T_i(u):=W_i(u)\gamma(u_i)$.
The mean objective is
\begin{equation}
	\label{eq:new-C-mean}
	F_{\rm C}(x,y)
	=
	\alpha\left[
	H_m(u)
	-\frac12\sum_{i=1}^{m}(v^i)^\top B_2v^i
	+\frac1{\sqrt2}\sum_{i=1}^{m}
	\bigl(W_i(u)v_1^i-T_i(u)v_2^i\bigr)
	\right],
	\ 
	B_2:=
	\begin{pmatrix}
		3/2&-1\\
		-1&3/2
	\end{pmatrix}.
\end{equation}

Let $\omega\sim\operatorname{Bernoulli}(p)$, where
$p:=1$ if $\sigma=0$, and
$p:=(2\alpha\zeta)^2/((2\alpha\zeta)^2+\sigma^2)$
otherwise.
Define the stochastic joint gradient
$G_{\rm C}(x,y;\omega):=(G_x(x,y;\omega),G_y(x,y;\omega))$
by
\begin{equation}
	\label{eq:new-C-gradient}
	\begin{aligned}
		G_x(x,y;\omega)
		&=
		\frac{\alpha}{\beta}\left[
		\nabla H_m(u)
		+\frac1{\sqrt2}\sum_{i=1}^{m}
		\bigl(v_1^i\nabla W_i(u)-v_2^i\nabla T_i(u)\bigr)
		\right],\\
		G_y(x,y;\omega)
		&=
		\alpha\zeta\,
		\bigl(g_1(u,v;\omega),\ldots,g_m(u,v;\omega)\bigr),
	\end{aligned}
\end{equation}
where
\begin{equation*}
	g_i(u,v;\omega):=
	\begin{pmatrix}
		-\frac32v_1^i+\frac{\omega}{p}v_2^i
		+\frac{W_i(u)}{\sqrt2}\\[2pt]
		-\frac32v_2^i+\frac{\omega}{p}v_1^i
		-\frac{T_i(u)}{\sqrt2}
	\end{pmatrix}.
\end{equation*}
The oracle returns the exact function value together with
$G_{\rm C}$, as in Definition~\ref{def:stochastic-oracle}.
The primal gradient is exact; randomness affects only the
off-diagonal dual coupling.
Each round uses a fresh Bernoulli sample shared by all queries
in that round.

\begin{lemma}[Admissibility and stationarity barrier]
	\label{lem:new-C-instance}
	There exists a universal constant $c_0>0$ such that the following
	holds for all $L,\Delta,D_Y>0$, $\bar L\ge L$,
	$0\le\sigma\le D_Y\bar L/4$, and
	$0<\varepsilon\le
	c_0\min\{\sqrt{L\Delta},LD_Y\}$:
	\begin{enumerate}[label=(\roman*),leftmargin=*]
		\item The objective $F_{\rm C}$ is jointly $L$-smooth and
		$\mu$-strongly concave in $y$, with
		$\mu=\alpha\zeta^2/2$.
		Its value function satisfies
		$\Phi_{\rm C}(x)=\alpha\varphi_m(x/\beta)$ and
		$\Phi_{\rm C}(0)-\inf_x\Phi_{\rm C}(x)\le\Delta$.
		In particular,
		$F_{\rm C}\in\mathcal F_{\rm NC\text{-}C}(L,\Delta,D_Y)$.
		
		\item The stochastic gradient is unbiased, has variance at
		most $\sigma^2$, and satisfies
		\begin{equation*}
			\mathbb E_\omega
			\|G_{\rm C}(z;\omega)-G_{\rm C}(z';\omega)\|^2
			\le\bar L^2\|z-z'\|^2,
			\qquad z,z'\in X\times Y,
		\end{equation*}
		with the same sample used at both points.
		
		\item If $x_m=0$, then $S_{\rm OS}(x)>3\varepsilon/2$.
	\end{enumerate}
\end{lemma}

The proof is given in Appendix~\ref{app:ncc:construction}.
Although the constructed objective is strongly concave in its dual
variable, its strong-concavity parameter depends on the target
accuracy. It is therefore an admissible hard instance for the NC--C
class, which imposes no fixed positive strong-concavity parameter.

\subsection{Coordinate revelation}

The next lemma supplies the propagation property needed to apply
Lemma~\ref{lem:bernoulli-progress}.

\begin{lemma}[A successful-round barrier]
	\label{lem:ncc-coordinate-progress}
	For the oracle in \eqref{eq:new-C-gradient}, any zero-respecting
	algorithm must observe at least $m$ successful rounds before its
	output can have a nonzero last primal coordinate.
	This statement holds for arbitrary finite batch sizes, provided that
	all queries in a round are chosen before that round's responses.
\end{lemma}

\begin{proof}
	Let $R_t:=\sum_{s=1}^{t}\omega_s$ be the number of successful rounds
	among the first $t$ rounds.
	We prove the following support invariant: after round $t$, the
	revealed coordinates are contained in
	\begin{equation*}
		\{x_i:i\le R_t\}
		\ \cup\
		\{y_2^i:i\le R_t\}
		\ \cup\
		\{y_1^i:i\le R_t+1\},
	\end{equation*}
	with indices truncated at $m$.
	It holds initially because no coordinates have been revealed.
	
	Suppose it holds after round $t-1$, and write $r:=R_{t-1}<m$.
	Every query in round $t$ then satisfies
	$u_i=0$ for $i>r$,
	$v_2^i=0$ for $i>r$, and
	$v_1^i=0$ for $i>r+1$.
	At such a query,
	$W_i(u)=T_i(u)=0$ for $i>r+1$, while
	$T_{r+1}(u)=0$ because $\gamma(0)=0$.

	Moreover,
	$\operatorname{supp}(\nabla W_{r+1}(u))\subseteq[r]$:
	the function depends only on the first $r+1$ coordinates, and its
	derivative with respect to $u_{r+1}$ vanishes because
	$S'(u_{r+1})=S'(0)=0$.
	For $i\le r$, both $W_i$ and $T_i$ depend only on the first $i$
	coordinates.
	Consequently, the primal response cannot reveal any coordinate
	$x_i$ with $i>r$: the only potentially active block beyond $r$
	has coefficient $v_2^{r+1}=0$, and its $W_{r+1}$ term has no
	derivative beyond $r$.
	
	For the dual response, the first component can reveal at most
	$y_1^{r+1}$.
	The second component in block $r+1$ equals
	$\omega_t v_1^{r+1}/p$, because
	$v_2^{r+1}=T_{r+1}(u)=0$.
	Thus a failed round reveals no new second dual coordinate,
	whereas a successful round can reveal at most $y_2^{r+1}$.
	No later block can be revealed.
	
	These facts preserve the invariant when $\omega_t=0$ and imply it
	with $R_t=r+1$ when $\omega_t=1$.
	The case $r\ge m$ is immediate.
	The argument applies simultaneously to every query in the batch,
	because all queries use the same previously revealed support.
	
	By the zero-respecting output restriction, a nonzero last primal
	coordinate therefore requires $R_t\ge m$.
	Exact function values do not affect this conclusion, since they do
	not enlarge the revealed-coordinate set.
\end{proof}

\subsection{Complexity lower bound}

We first state the lower bound furnished directly by the construction.

\begin{proposition}[Lower bound for the constructed instance]
	\label{prop:ncc-stochastic-instance}
	Under the conditions of Lemma~\ref{lem:new-C-instance}, there is a
	universal constant $c_{\rm s}>0$ such that every zero-respecting
	algorithm using at most
	\begin{equation*}
		c_{\rm s}\left(
		\frac{L\Delta}{\varepsilon^2}
		+
		\frac{L^3D_Y^2\Delta\sigma^2}{\varepsilon^6}
		\right)
	\end{equation*}
	oracle rounds on the constructed instance satisfies
	$\mathbb E[S_{\rm OS}(\widehat x)]>\varepsilon$ and hence
	$\mathbb E[S_{\rm OS}(\widehat x)^2]>\varepsilon^2$.
\end{proposition}

\begin{proof}
	By Lemma~\ref{lem:ncc-coordinate-progress}, revealing the last
	primal coordinate requires at least $m$ successful rounds.
	Lemma~\ref{lem:bernoulli-progress} therefore gives
	$\Pr(\widehat x_m=0)\ge1-Np/m\ge3/4$
	whenever $N\le m/(4p)$.
	Together with Lemma~\ref{lem:new-C-instance}(iii), this yields
	\begin{equation*}
		\mathbb E[S_{\rm OS}(\widehat x)]
		>
		\frac{3\varepsilon}{2}\Pr(\widehat x_m=0)
		\ge\frac{9\varepsilon}{8}
		>\varepsilon.
	\end{equation*}
	The second-moment conclusion follows from Jensen's inequality.
	
	It remains to bound $m/(4p)$ from below.
	The parameter choices give
	$\alpha=72C_*\varepsilon^2/(c_b^2L)$ and
	$1/p=1+D_Y^2\sigma^2/(64\alpha^2)$.
	Choose $c_0$ sufficiently small that
	$\Delta/(10\alpha)\ge2$.
	Then $m\ge\Delta/(20\alpha)$, and
	\begin{equation*}
		\frac{m}{4p}
		\ge
		\frac{\Delta}{80\alpha}
		\left(1+\frac{D_Y^2\sigma^2}{64\alpha^2}\right)
		\ge
		c_{\rm s}\left(
		\frac{L\Delta}{\varepsilon^2}
		+
		\frac{L^3D_Y^2\Delta\sigma^2}{\varepsilon^6}
		\right)
	\end{equation*}
	for a universal constant $c_{\rm s}>0$.
\end{proof}

To obtain the additional deterministic term, we use a deterministic
NC--C lower bound of order
$L^2D_Y\Delta\varepsilon^{-3}$.
The required form of this external result is the following:
for the zero-respecting randomized batch-access algorithms defined
in Section~\ref{common:models}, fewer than
$c_{\rm d}L^2D_Y\Delta\varepsilon^{-3}$ rounds fail to attain
$\mathbb E[S_{\rm OS}(\widehat x)]\le\varepsilon$
on some admissible deterministic instance, where $c_{\rm d}>0$
is universal.


The deterministic construction of \citet{Zhang2026} is used here
in this form, subject to the access-model verification described
above. Its primal domain is unconstrained, and its dual domain
contains the origin and has diameter at most $D_Y$; hence its dual
radius is at most $D_Y$.
Equipping this instance with its exact gradient oracle gives zero
variance and $\bar L$-mean-square smoothness whenever $\bar L\ge L$.

\begin{theorem}[Stochastic NC--C lower bound]
	\label{nccmss:thm:zr}
	Assume that the deterministic lower bound stated above applies to
	the algorithm class in
	Definition~\ref{def:zero-respecting-algorithm}.
	There exist universal constants $c,c_0>0$ such that the following
	holds for all parameters satisfying
	\[
	L,\Delta,D_Y>0,\qquad
	\bar L\ge L,\qquad
	0\le\sigma\le\frac{D_Y\bar L}{4},\qquad
	0<\varepsilon\le
	c_0\min\{\sqrt{L\Delta},LD_Y\}.
	\]
	For every zero-respecting randomized algorithm with an oracle-round
	budget of at most
	\[
	c\left(
	\frac{L^2D_Y\Delta}{\varepsilon^3}
	+
	\frac{L^3D_Y^2\Delta\sigma^2}{\varepsilon^6}
	\right),
	\]
	there exist an objective
	$F\in\mathcal F_{\rm NC\text{-}C}(L,\Delta,D_Y)$
	and an unbiased stochastic first-order oracle for $F$ with variance
	at most $\sigma^2$ and mean-square-smoothness constant at most
	$\bar L$ such that the algorithm's output $\widehat x$ satisfies
	\[
	\mathbb E[S_{\rm OS}(\widehat x)]>\varepsilon.
	\]
	The expectation is over both algorithmic and oracle randomness.
	Consequently,
	$\mathbb E[S_{\rm OS}(\widehat x)^2]>\varepsilon^2$.
\end{theorem}

\begin{proof}
	Write
	$A:=L^2D_Y\Delta\varepsilon^{-3}$ and
	$B:=L^3D_Y^2\Delta\sigma^2\varepsilon^{-6}$.
	The deterministic lower bound supplies an admissible instance
	on which a budget of at most $c_{\rm d}A$ rounds fails.
	Proposition~\ref{prop:ncc-stochastic-instance} supplies an
	admissible stochastic instance on which a budget of at most
	$c_{\rm s}B$ rounds fails.
	Choose $c:=\frac12\min\{c_{\rm d},c_{\rm s}\}$.
	If $A\ge B$, then $c(A+B)\le c_{\rm d}A$, and the deterministic
	instance gives the conclusion.
	If $B>A$, then $c(A+B)\le c_{\rm s}B$, and the stochastic instance
	gives the conclusion.
	The second-moment statement again follows from Jensen's inequality.
\end{proof}

Since each oracle round contains at least one gradient evaluation,
the round lower bounds also imply lower bounds of the same order
on the number of individual gradient evaluations.

\section{Stochastic NC--SC lower bounds}
\label{sec:ncsc}

We establish stochastic lower bounds for NC--SC problems under
bounded-variance and mean-square-smooth oracle models.
Both constructions use an unconstrained primal domain and the
stationarity criterion of
Definition~\ref{def:optimization-stationarity}.

For $L,\mu,\Delta>0$, let
$\mathcal F_{\rm NC\text{-}SC}(L,\mu,\Delta)$ consist of continuously
differentiable objectives
$F:\mathbb R^m\times Y\to\mathbb R$, over arbitrary finite
dimensions, where $Y$ is a nonempty closed convex set containing
the origin. We require joint $L$-smoothness, the initial-gap bound
$\Phi(0)-\inf_x\Phi(x)\le\Delta$, and
\[
F(x,y')
\le
F(x,y)+\langle\nabla_yF(x,y),y'-y\rangle
-\frac{\mu}{2}\|y'-y\|^2,
\qquad x\in\mathbb R^m,\quad y,y'\in Y.
\]
Here $\Phi(x):=\max_{y\in Y}F(x,y)$ and $\kappa:=L/\mu$.
The maximum is attained uniquely: strong concavity gives a coercive
quadratic upper bound in the dual variable, and $Y$ is closed.
No uniform dual-radius bound is imposed on this class.
The BV construction uses a compact dual ball, whereas the MSS
construction uses $Y=\mathbb R^m$.

Throughout this section, one oracle round consists of a finite
batch of queries chosen before any response from that round.
All queries in a round share one fresh Bernoulli sample.
The oracle also returns exact function values, as specified in
Section~\ref{common:models}.

\paragraph{Deterministic baseline.}
The combined bounds below use a deterministic NC--SC lower bound
of order $L\Delta\sqrt{\kappa}\,\varepsilon^{-2}$.
Specifically, we require this bound for the randomized
zero-respecting algorithms of
Definition~\ref{def:zero-respecting-algorithm}, measured in oracle
rounds and with failure criterion
$\mathbb E[S_{\rm OS}(\widehat x)]>\varepsilon$.
We refer to this statement as the deterministic baseline.

The result of \citet{li} is a candidate for this baseline.
Its applicability to the present stationarity criterion and
batch-access model must be verified before invoking the combined
theorems below. The stochastic bounds proved directly from our
constructions do not depend on this external result.

\subsection{Bounded-variance oracles}
\label{sec:bv}

\paragraph{Construction.}
Let $G=(G_1,\ldots,G_m)$ be the feature map of the common chain,
so that $\varphi_m=H_m+\|G\|^2$.
Define
\[
F_{\rm BV}(x,y)
:=
\alpha H_m(x/\beta)
+\sqrt{2\mu\alpha}\,\langle G(x/\beta),y\rangle
-\frac{\mu}{2}\|y\|^2
\]
on $X:=\mathbb R^m$ and
$Y:=\{y\in\mathbb R^m:\|y\|\le2\sqrt{\alpha/\mu}\}$.
Set
\[
\beta:=\frac{16C_{\rm B}\varepsilon}{c_bL},
\qquad
\alpha:=\frac{L\beta^2}{C_{\rm B}},
\qquad
m:=\left\lfloor\frac{\Delta}{5\alpha}\right\rfloor,
\qquad
p:=\frac{C_{\rm V}\mu\alpha}{C_{\rm V}\mu\alpha+\sigma^2},
\]
where $C_{\rm B},C_{\rm V}>0$ are the universal constants specified
in Appendix~\ref{app:bv:construction}.

Let $\chi$ be the bounded $C^{1,1}$ clipping function defined in
Appendix~\ref{app:bv:construction}, and set
$\chi_\beta(t):=\beta\chi(t/\beta)$.
For vectors, $\chi_\beta$ is applied coordinatewise.
With $\omega\sim\operatorname{Bernoulli}(p)$, define
\[
G_{\rm BV}(x,y;\omega)
:=
\nabla F_{\rm BV}(x,y)
+
\left(\frac{\omega}{p}-1\right)\sqrt{2\mu\alpha}
\begin{pmatrix}
	\beta^{-1}\nabla G(x/\beta)^\top\chi_\beta(y)\\
	\bigl(G_i(x/\beta)\chi_\beta'(y_i)\bigr)_{i=1}^{m}
\end{pmatrix}.
\]
The clipping bounds the noise uniformly without changing its
coordinate-revelation behavior at zero.

\begin{lemma}[Properties of the BV construction]
	\label{lem:bv-class}
	\label{lem:bv-oracle}
	There exists a universal constant $c_0>0$ such that, for all
	$L,\mu,\Delta>0$, $\sigma\ge0$, $\kappa=L/\mu\ge2$, and
	$0<\varepsilon\le c_0\sqrt{L\Delta}$, the following hold:
	\begin{enumerate}[label=(\roman*),leftmargin=*]
		\item
		$F_{\rm BV}\in
		\mathcal F_{\rm NC\text{-}SC}(L,\mu,\Delta)$.
		
		\item
		The value function is
		\begin{equation}
			\label{eq:bv-value}
			\Phi_{\rm BV}(x)=\alpha\varphi_m(x/\beta).
		\end{equation}
		
		\item
		The stochastic gradient $G_{\rm BV}$ is unbiased and has
		variance at most $\sigma^2$.
		
		\item
		If a query is supported on the first $r$ coordinate pairs
		$(x_i,y_i)$, its response is supported on the first
		$\min\{r+1,m\}$ pairs. If $\omega=0$, no pair beyond $r$
		is revealed.
		
		Consequently, for every zero-respecting algorithm,
		$\widehat x_m=0$ whenever fewer than $m$ rounds are successful.
		This conclusion holds for arbitrary finite nonadaptive
		same-sample batches.
	\end{enumerate}
\end{lemma}

The proof is given in Appendix~\ref{app:bv:construction}.

\begin{proposition}[Lower bound for the BV construction]
	\label{prop:bv-stochastic}
	Under the conditions of Lemma~\ref{lem:bv-class}, there exists a
	universal constant $c_{\rm B}'>0$ such that every zero-respecting
	randomized algorithm using at most
	\[
	c_{\rm B}'\left(
	\frac{L\Delta}{\varepsilon^2}
	+
	\frac{L\Delta\kappa\sigma^2}{\varepsilon^4}
	\right)
	\]
	oracle rounds on the constructed instance satisfies
	$\mathbb E[S_{\rm OS}(\widehat x)]>\varepsilon$.
	Consequently,
	$\mathbb E[S_{\rm OS}(\widehat x)^2]>\varepsilon^2$.
\end{proposition}
\begin{proof}
	Choose $c_0\le c_b/\sqrt{10240C_{\rm B}}$.
	Then $\alpha=256C_{\rm B}\varepsilon^2/(c_b^2L)\le\Delta/40$
	and $m\ge\Delta/(10\alpha)$.
	Since $C_{\rm B}\ge2C_\varphi$, the value function satisfies
	$\operatorname{Lip}(\nabla\Phi_{\rm BV})
	\le\alpha C_\varphi/\beta^2\le L/2$.
	Set $b:=8\varepsilon$.
	The parameter choices give $\alpha c_b/(2\beta)=b$ and
	$b/(L\beta)=c_b/(2C_{\rm B})\le1$.
	Lemma~\ref{nccmss:lem:moreau-transfer} therefore yields
	$S_{\rm OS}(x)>8\varepsilon$ whenever $x_m=0$.
	
	By Lemma~\ref{lem:bv-class}(iv), revealing $x_m$ requires at
	least $m$ successful rounds.
	For $N\le m/(2p)$, Lemma~\ref{lem:bernoulli-progress} gives
	$\Pr(\widehat x_m=0)\ge1/2$.
	Hence $\mathbb E[S_{\rm OS}(\widehat x)]
	>8\varepsilon\,\Pr(\widehat x_m=0)\ge4\varepsilon>\varepsilon$.
	The second-moment conclusion follows from Jensen's inequality.
	
	Finally, $p^{-1}=1+\sigma^2/(C_{\rm V}\mu\alpha)$ and
	$m\ge\Delta/(10\alpha)$ imply
	$m/(2p)\ge\Delta/(20\alpha)
	+\Delta\sigma^2/(20C_{\rm V}\mu\alpha^2)$.
	Substituting $\alpha=256C_{\rm B}\varepsilon^2/(c_b^2L)$
	and $\mu=L/\kappa$, we obtain
	$m/(2p)\ge c_{\rm B}'(L\Delta/\varepsilon^2
	+L\Delta\kappa\sigma^2/\varepsilon^4)$ for any
	$c_{\rm B}'>0$ no larger than
	$\min\{c_b^2/(5120C_{\rm B}),
	c_b^4/(20\cdot256^2C_{\rm V}C_{\rm B}^2)\}$.
	This proves the claim.
\end{proof}

\begin{theorem}[Combined BV lower bound]
	\label{thm:intro-bv}
	Assume that the deterministic baseline stated above holds for
	the algorithm class in
	Definition~\ref{def:zero-respecting-algorithm}.
	There exist universal constants $\kappa_0\ge2$ and $c,c_0>0$
	such that the following holds for all parameters satisfying
	\[
	L,\mu,\Delta>0,\qquad
	\kappa:=L/\mu\ge\kappa_0,\qquad
	\sigma\ge0,\qquad
	0<\varepsilon\le c_0\sqrt{L\Delta}.
	\]
	For every zero-respecting randomized algorithm using at most
	\begin{equation}
		\label{eq:bv-rate}
		c\left(
		\frac{L\Delta\sqrt{\kappa}}{\varepsilon^2}
		+
		\frac{L\Delta\kappa\sigma^2}{\varepsilon^4}
		\right)
	\end{equation}
	oracle rounds, there exist an objective
	$F\in\mathcal F_{\rm NC\text{-}SC}(L,\mu,\Delta)$
	and an unbiased stochastic first-order oracle for $F$ with
	variance at most $\sigma^2$ such that the algorithm's output
	$\widehat x$ satisfies
	$\mathbb E[S_{\rm OS}(\widehat x)]>\varepsilon.$
	Consequently,
	$\mathbb E[S_{\rm OS}(\widehat x)^2]>\varepsilon^2$.
	Both expectations are over algorithmic and oracle randomness.
\end{theorem}

\begin{proof}
	Let
	$A:=L\Delta\sqrt{\kappa}\,\varepsilon^{-2}$ and
	$B:=L\Delta\kappa\sigma^2\varepsilon^{-4}$.
	The deterministic baseline gives a lower bound $c_{\rm d}A$
	for a universal $c_{\rm d}>0$.
	Its exact gradient oracle has zero variance and is therefore
	admissible in the BV model.
	Proposition~\ref{prop:bv-stochastic} gives a lower bound
	$c_{\rm B}'B$ on the constructed stochastic instance.
	Set $c:=\frac12\min\{c_{\rm d},c_{\rm B}'\}$.
	If $A\ge B$, then $c(A+B)\le c_{\rm d}A$, and the
	deterministic instance gives the conclusion.
	Otherwise, $c(A+B)\le c_{\rm B}'B$, and the stochastic instance
	gives the conclusion.
\end{proof}

\subsection{Mean-square-smooth oracles}
\label{sec:mss}

\paragraph{Construction.}
Take $X=Y=\mathbb R^m$ and define
\[
F_{\rm MS}(x,y)
:=
\delta\lambda\varphi_m(y/\lambda)-\mu\|y-ax\|^2.
\]
Its stochastic gradient is
\[
G_{\rm MS}(x,y;\omega)
:=
\nabla F_{\rm MS}(x,y)
+
\left(\frac{\omega}{p}-1\right)
\begin{pmatrix}
	0\\
	\delta\nabla Q_m(y/\lambda)
\end{pmatrix},
\qquad
\omega\sim\operatorname{Bernoulli}(p),
\]
where
\begin{equation}
	\label{eq:mss-scales-en}
	\begin{aligned}
		a&:=\sqrt{\frac{\kappa}{8}},
		&
		\delta&:=\frac{4\varepsilon}{ac_b},
		&
		C_\lambda&:=\frac{2(C_\varphi+C_Q)}{V_0},\\
		\lambda&:=\frac{2C_\lambda\sigma}{\bar L},
		&
		p&:=\left(\frac{\delta V_0}{\sigma}\right)^2,
		&
		m&:=\left\lfloor\frac{\Delta}{10\delta\lambda}\right\rfloor.
	\end{aligned}
\end{equation}
Here $C_Q$ and $V_0$ are the constants from
Appendix~\ref{app:common}.
The noise acts only on the dual gradient.
Let
$h(y):=\delta\lambda\varphi_m(y/\lambda)$ and
$\ell:=\delta C_\varphi/\lambda$; then $h$ is $\ell$-smooth.

\begin{lemma}[Properties of the MSS construction]
	\label{lem:mss-class}
	\label{lem:mss-progress}
	\label{lem:mss-risk}
	Let $L,\mu,\Delta>0$, $\bar L\ge L$, $\sigma>0$, and
	$\kappa=L/\mu\ge8$.
	Suppose that the parameters in \eqref{eq:mss-scales-en} satisfy
	$p\le1$, $\delta\lambda\le\Delta/20$, and $\ell\le\mu$.
	Then:
	\begin{enumerate}[label=(\roman*),leftmargin=*]
		\item
		$F_{\rm MS}\in
		\mathcal F_{\rm NC\text{-}SC}(L,\mu,\Delta)$.
		
		\item
		The stochastic gradient is unbiased, has variance at most
		$\sigma^2$, and satisfies the $\bar L$-MSS condition.
		
		\item
		Revealing $x_m$ requires at least $m$ successful rounds.
		This statement holds for arbitrary finite nonadaptive
		same-sample batches.
		
		\item
		The value function satisfies
		\begin{equation}
			\label{eq:mss-transfer-upper}
			\|\nabla h(ax)\|
			\le
			\frac1a\left(1+\frac{\ell}{2\mu}\right)
			\|\nabla\Phi_{\rm MS}(x)\|.
		\end{equation}
		In particular, $x_m=0$ implies
		$\|\nabla\Phi_{\rm MS}(x)\|\ge8\varepsilon/3$.
	\end{enumerate}
\end{lemma}
The proof is given in Appendix~\ref{app:mss:construction}.

We next relate this gradient barrier to the common
Moreau-envelope stationarity criterion.

\begin{lemma}[MSS stationarity barrier]
	\label{lem:mss-moreau-barrier}
	Under the conditions of Lemma~\ref{lem:mss-class}, the value
	function $\Phi_{\rm MS}$ is $L/4$-smooth on $\mathbb R^m$.
	Moreover, every $x\in\mathbb R^m$ with $x_m=0$ satisfies
	\[
	S_{\rm OS}(x)\ge\frac{64}{27}\varepsilon.
	\]
\end{lemma}

\begin{proof}
	Let $y^*(x)$ be the unique dual maximizer.
	Its first-order condition and Danskin's theorem give
	\[
	\nabla h(y^*(x))=2\mu(y^*(x)-ax),
	\qquad
	\nabla\Phi_{\rm MS}(x)
	=
	2\mu a(y^*(x)-ax)
	=
	a\nabla h(y^*(x)).
	\]
	Subtracting the first-order conditions at $x$ and $x'$ yields
	\[
	(2\mu-\ell)\|y^*(x)-y^*(x')\|
	\le2\mu a\|x-x'\|.
	\]
	Since $\ell\le\mu$,
	\[
	\operatorname{Lip}(\nabla\Phi_{\rm MS})
	\le
	\frac{2\mu a^2\ell}{2\mu-\ell}
	\le2a^2\ell
	\le\frac L4.
	\]
	Let $x^+$ minimize
	$u\mapsto\Phi_{\rm MS}(u)+L\|u-x\|^2$.
	The first-order condition gives
	$\nabla\Phi_{\rm MS}(x^+)=2L(x-x^+)$.
	Therefore,
	\[
	\begin{aligned}
		\|\nabla\Phi_{\rm MS}(x)\|
		&\le
		\|\nabla\Phi_{\rm MS}(x^+)\|
		+\frac L4\|x-x^+\|=
		\frac98 S_{\rm OS}(x).
	\end{aligned}
	\]
	If $x_m=0$, Lemma~\ref{lem:mss-class}(iv) now gives
	$S_{\rm OS}(x)\ge(8/9)(8\varepsilon/3)
	=64\varepsilon/27$.
\end{proof}

\begin{proposition}[Lower bound for the MSS construction]
	\label{prop:mss-stochastic}
	There exist universal constants $c_{\rm M}',c_1>0$ such that,
	for every $L,\mu,\Delta>0$, $\bar L\ge L$, $\sigma>0$, and
	$\kappa=L/\mu\ge8$ satisfying
	\[
	0<\varepsilon
	\le c_1\sqrt{\kappa}
	\min\left\{
	\sigma,\frac{\Delta\bar L}{\sigma},
	\frac{\mu\sigma}{\bar L}
	\right\},
	\]
	every zero-respecting randomized algorithm using at most
	$c_{\rm M}'
	\frac{\Delta\bar L\sigma\kappa^{3/2}}{\varepsilon^3}$
	oracle rounds on the constructed instance satisfies
	$\mathbb E[S_{\rm OS}(\widehat x)]>\varepsilon$.
	Consequently,
	$\mathbb E[S_{\rm OS}(\widehat x)^2]>\varepsilon^2$.
\end{proposition}

\begin{proof}
	Choose
	\[
	c_1\le
	\min\left\{
	\frac{c_b}{4\sqrt8V_0},
	\frac{c_b}{160\sqrt8C_\lambda},
	\frac{c_bC_\lambda}{2\sqrt8C_\varphi}
	\right\}.
	\]
	Since $\delta=4\sqrt8\,\varepsilon/(c_b\sqrt{\kappa})$,
	the three accuracy restrictions imply, respectively,
	\[
	\frac{\delta V_0}{\sigma}\le1,
	\qquad
	\delta\lambda\le\frac{\Delta}{20},
	\qquad
	\ell\le\mu.
	\]
	Thus Lemma~\ref{lem:mss-class} applies, and
	$m\ge\Delta/(20\delta\lambda)\ge1$.
	If $N\le m/(2p)$, the coordinate-revelation property and
	Lemma~\ref{lem:bernoulli-progress} give
	$\Pr(\widehat x_m=0)\ge1/2$.
	Lemma~\ref{lem:mss-moreau-barrier} then yields
	\[
	\mathbb E[S_{\rm OS}(\widehat x)]
	\ge
	\frac{64\varepsilon}{27}\Pr(\widehat x_m=0)
	\ge\frac{32\varepsilon}{27}
	>\varepsilon.
	\]
	The second-moment conclusion follows from Jensen's inequality.
	Finally,
	\begin{align*}
		\frac{m}{2p}
		&\ge\frac{\Delta}{40\delta\lambda p}=
		\frac{\Delta\bar L\sigma}
		{80C_\lambda V_0^2\delta^3}=
		\frac{c_b^3\Delta\bar L\sigma\kappa^{3/2}}
		{40960\sqrt8C_\lambda V_0^2\varepsilon^3}.
	\end{align*}
	This proves the claim for a sufficiently small universal
	$c_{\rm M}'>0$.
\end{proof}

\begin{theorem}[Combined MSS lower bound]
	\label{thm:mss-sum}
	Assume that the deterministic baseline stated above holds.
	There exist universal constants $\kappa_0,c,c_0,c_1>0$ such that
	the following holds for all $L,\mu,\Delta>0$, $\bar L\ge L$,
	$\sigma>0$, and $\kappa=L/\mu\ge\kappa_0$ satisfying
	\begin{equation}
		\label{eq:mss-range-en}
		0<\varepsilon\le c_0\sqrt{L\Delta},
		\qquad
		\varepsilon\le c_1\sqrt{\kappa}
		\min\left\{
		\sigma,\frac{\Delta\bar L}{\sigma},
		\frac{\mu\sigma}{\bar L}
		\right\}.
	\end{equation}
	For every zero-respecting randomized algorithm with an
	oracle-round budget of at most
	\begin{equation}
		\label{eq:mss-rate}
		c\left(
		\frac{L\Delta\sqrt{\kappa}}{\varepsilon^2}
		+
		\frac{\Delta\bar L\sigma\kappa^{3/2}}{\varepsilon^3}
		\right),
	\end{equation}
	there exist an objective
	$F\in\mathcal F_{\rm NC\text{-}SC}(L,\mu,\Delta)$ and an
	unbiased stochastic first-order oracle with variance at most
	$\sigma^2$ and mean-square-smoothness constant at most $\bar L$
	such that
	\[
	\mathbb E[S_{\rm OS}(\widehat x)]>\varepsilon.
	\]
	In particular,
	$\mathbb E[S_{\rm OS}(\widehat x)^2]>\varepsilon^2$.
	The expectations include both algorithmic and oracle randomness.
\end{theorem}

\begin{proof}
	Choose $\kappa_0$ large enough for the deterministic baseline and
	$\kappa_0\ge8$, and choose $c_0,c_1$ to satisfy the requirements
	of that baseline and Proposition~\ref{prop:mss-stochastic}.
	The exact gradient oracle of any jointly $L$-smooth objective
	has zero variance and satisfies the $\bar L$-MSS condition because
	$\bar L\ge L$. Hence the deterministic baseline is admissible
	in the present oracle model.
	Let
	$A:=L\Delta\sqrt{\kappa}\,\varepsilon^{-2}$ and
	$B:=\Delta\bar L\sigma\kappa^{3/2}\varepsilon^{-3}$.
	The deterministic baseline gives a lower bound $c_{\rm d}A$,
	while Proposition~\ref{prop:mss-stochastic} gives a lower bound
	$c_{\rm M}'B$.
	Taking $c:=\frac12\min\{c_{\rm d},c_{\rm M}'\}$ and selecting
	the instance corresponding to the larger of $A$ and $B$
	proves \eqref{eq:mss-rate}.
\end{proof}

Both combined bounds are worst-case statements over admissible
objective--oracle pairs; the deterministic and stochastic terms
need not be attained by the same instance.
Since every oracle round contains at least one gradient evaluation,
the round lower bounds also imply lower bounds of the same order
on the number of individual gradient evaluations.

\section{Conclusion}\label{sec:conclusion}

This work establishes complexity barriers for stochastic first-order minimax optimization, with a primary focus on NC--C problems when variance reduction is allowed.
Our main result shows that the
$\varepsilon^{-6}$ stochastic barrier persists under mean-squared
smoothness for zero-respecting randomized algorithms with same-sample oracle access.
Its significance lies not in a new accuracy exponent, but in extending the lower bound to an algorithm class that permits variance reduction.
The complementary NC--SC results capture the dependence on accuracy and conditioning under bounded variance and mean-squared smoothness.

The common product-weighted chain separates the difficulty of attaining value-function stationarity from the delay in revealing primal--dual information.
This separation enables hard instances across different concavity and stochastic regularity regimes, while preserving the coupled oracle access used by variance-reduction methods.
Together, the results show that stronger stochastic regularity does not by itself guarantee an improved minimax accuracy dependence: the interaction between primal and dual variables remains a fundamental source of difficulty within the algorithm class considered here.

An important next step is to determine whether these barriers extend beyond zero-respecting algorithms.
Resolving this question would clarify the extent to which the observed complexity is intrinsic to stochastic minimax optimization rather than tied to coordinate-revelation restrictions.

\acks{This work was supported by the National Key R\&D Program
of China (Grant Nos.~2025YFA1017801 and 2025YFA1017800) and the National Natural Science Foundation of China (Grant No.~12471294).}

\appendix
\section{Properties of the common zero-chain}
\label{app:common}

This appendix proves Lemma~\ref{lem:common-chain}.
Throughout, let $q:=2\pi/3$, write
$P_i(x):=\prod_{j=1}^{i}S(x_j)$ with $P_0(x):=1$, and recall that
$W_i(x)=P_{i-1}(x)(1-S(x_i))$.
Empty products and sums are understood as $1$ and $0$, respectively.

Since $0\le S\le1$ and $W_i=P_{i-1}-P_i$, the weights satisfy
\begin{equation}
	\label{eq:weight-sum}
	W_i(x)\ge0,
	\qquad
	\sum_{i=1}^{m}W_i(x)=1-P_m(x)\le1.
\end{equation}
Direct differentiation of $S$ gives
\begin{equation}
	\label{eq:S-bounds-for-weights}
	S'(u)^2=4q^2S(u)(1-S(u)),
	\qquad
	|S''(u)|\le2q^2.
\end{equation}
The first identity holds everywhere, while the second holds wherever
$S''$ exists. In particular, $S'$ is globally $2q^2$-Lipschitz.
The functions $S$, $\gamma$, and $J$ are continuously differentiable
and piecewise twice continuously differentiable.

\begin{proof}[Proof of Lemma~\ref{lem:common-chain}]
	\paragraph{(i) Dimension-independent smoothness.}
	We first bound the derivatives of weighted sums of $W_i$, then use
	these estimates to control $G$ and $Q_m$.
	
	\emph{Step 1: The separable term.}
	Recall that
	$H_m(x)=\sum_{i=1}^{m}(-4S(x_i)+J(x_i))$.
	Since $S'$ is $2q^2$-Lipschitz and $J'$ is $1$-Lipschitz, the
	separable structure yields
	\begin{equation}
		\label{eq:H-gradient-bound}
		\|\nabla H_m(x)-\nabla H_m(x')\|
		\le H_0\|x-x'\|,
		\qquad
		H_0:=8q^2+1=\frac{32\pi^2}{9}+1.
	\end{equation}
	
	\emph{Step 2: Weighted derivative estimates.}
	Fix $a\in\mathbb R^m$ and define
	$f_a(x):=\sum_{i=1}^{m}a_iW_i(x)$.
	For $r\in[m]$, let
	\begin{equation*}
		B_r(x):=
		-a_r+
		\sum_{i=r+1}^{m}a_i
		\left(\prod_{j=r+1}^{i-1}S(x_j)\right)(1-S(x_i)).
	\end{equation*}
	The nonnegative weights in this sum add up to
	$1-\prod_{j=r+1}^{m}S(x_j)\le1$, so
	$|B_r(x)|\le2\|a\|_\infty$.
	Differentiating gives
	\begin{equation}
		\label{eq:weighted-first-derivative}
		\partial_r f_a(x)=P_{r-1}(x)S'(x_r)B_r(x).
	\end{equation}
	Consequently, by \eqref{eq:S-bounds-for-weights} and
	\eqref{eq:weight-sum},
	\begin{equation*}
		\begin{aligned}
			\|\nabla f_a(x)\|^2
			&\le
			16q^2\|a\|_\infty^2
			\sum_{r=1}^{m}P_{r-1}(x)^2S(x_r)(1-S(x_r))\\
			&\le
			16q^2\|a\|_\infty^2\sum_{r=1}^{m}W_r(x)
			\le16q^2\|a\|_\infty^2.
		\end{aligned}
	\end{equation*}
	Thus
	\begin{equation}
		\label{eq:weighted-first-bound}
		\left\|\sum_{i=1}^{m}a_i\nabla W_i(x)\right\|
		\le C_W\|a\|_\infty,
		\qquad C_W:=4q.
	\end{equation}
	
	We next bound the Hessian at points outside the piecewise joining
	hyperplanes. Since $B_r$ does not depend on $x_r$,
	$\partial_{rr}^2f_a=P_{r-1}S''(x_r)B_r$, and hence
	$|\partial_{rr}^2f_a|\le4q^2\|a\|_\infty$.
	For $r<s$, differentiation of
	\eqref{eq:weighted-first-derivative} gives
	\begin{equation*}
		\partial_{rs}^2f_a(x)
		=
		P_{r-1}(x)S'(x_r)
		\left(\prod_{j=r+1}^{s-1}S(x_j)\right)
		S'(x_s)B_s(x).
	\end{equation*}
	Using \eqref{eq:S-bounds-for-weights} and $0\le S\le1$,
	\begin{equation*}
		P_{r-1}(x)|S'(x_r)|
		\left(\prod_{j=r+1}^{s-1}S(x_j)\right)|S'(x_s)|
		\le4q^2\sqrt{W_r(x)W_s(x)}.
	\end{equation*}
	Indeed, after squaring and canceling the common factors, this
	inequality follows from
	$S(x_r)S(x_s)\prod_{j=r+1}^{s-1}S(x_j)\le1$;
	if a common factor vanishes, the inequality follows directly.
	It follows that
	$|\partial_{rs}^2f_a(x)|
	\le8q^2\|a\|_\infty\sqrt{W_r(x)W_s(x)}$.
	
	For any $h\in\mathbb R^m$, these estimates imply
	\begin{equation*}
		\begin{aligned}
			|h^\top\nabla^2f_a(x)h|
			&\le
			4q^2\|a\|_\infty\|h\|^2
			+
			8q^2\|a\|_\infty
			\left(\sum_{r=1}^{m}\sqrt{W_r(x)}\,|h_r|\right)^2\le
			12q^2\|a\|_\infty\|h\|^2,
		\end{aligned}
	\end{equation*}
	where the last inequality follows from Cauchy--Schwarz and
	\eqref{eq:weight-sum}.
	Thus $\|\nabla^2f_a(x)\|_{\rm op}\le12q^2\|a\|_\infty$
	on every smooth piece.
	
	Because $f_a$ is continuously differentiable across the joining
	hyperplanes, this Hessian bound gives a global gradient Lipschitz
	bound. To see this, integrate along any line segment not contained
	in a joining hyperplane, splitting it at its finitely many
	intersections with those hyperplanes. General segments follow by
	approximation and continuity of the gradient. Therefore,
	\begin{equation}
		\label{eq:weighted-second-bound}
		\left\|
		\sum_{i=1}^{m}a_i
		\bigl(\nabla W_i(x)-\nabla W_i(x')\bigr)
		\right\|
		\le D_W\|a\|_\infty\|x-x'\|,
		\qquad D_W:=12q^2.
	\end{equation}
	
	\emph{Step 3: Bounds for the weight map.}
	Let $W(x):=(W_1(x),\ldots,W_m(x))$.
	By duality between the $\ell_1$ and $\ell_\infty$ norms and
	\eqref{eq:weighted-first-bound},
	\begin{equation*}
		\begin{aligned}
			\|W(x)-W(x')\|_1
			&=
			\sup_{\|a\|_\infty\le1}|f_a(x)-f_a(x')|\le C_W\|x-x'\|.
		\end{aligned}
	\end{equation*}
	Moreover, \eqref{eq:weighted-second-bound} and
	$\|a\|_\infty\le\|a\|$ give
	\begin{equation}
		\label{eq:weight-derivatives}
		\|W(x)-W(x')\|_1\le C_W\|x-x'\|,
		\qquad
		\|\nabla W(x)-\nabla W(x')\|_{\rm op}
		\le D_W\|x-x'\|.
	\end{equation}
	Here and below, $\nabla W$ denotes the Jacobian with component
	gradients as its rows.
	
	\emph{Step 4: Bounds for the profile and the map $G$.}
	Since $|\gamma|\le1/2$, the profile satisfies
	$k_-\le\mathcal K(u)\le k_+$, where
	$k_-:=7/40$ and $k_+:=23/40$.
	Let $d(u):=\sqrt{\mathcal K(u)}$,
	$d_0:=\sqrt{k_+}$, and $d_1:=\|d'\|_\infty$.
	There is a universal constant $d_2>0$ such that
	$|d'(u)-d'(v)|\le d_2|u-v|$.
	
	For completeness, these constants are finite because
	$|\gamma'|\le1/4$ and $|\gamma''|\le1/8$ wherever the second
	derivative exists. Hence $\mathcal K'$ and the piecewise second
	derivative $\mathcal K''$ are uniformly bounded. The formulas
	\begin{equation*}
		d'(u)=\frac{\mathcal K'(u)}{2\sqrt{\mathcal K(u)}},
		\qquad
		d''(u)=
		\frac{\mathcal K''(u)}{2\sqrt{\mathcal K(u)}}
		-
		\frac{\mathcal K'(u)^2}{4\mathcal K(u)^{3/2}}
	\end{equation*}
	and $\mathcal K\ge k_->0$ then bound $d'$ and its piecewise
	derivative. Continuity of $d'$ across the joining points gives
	the asserted global Lipschitz bound.
	
	Recall that $G_i(x)=W_i(x)d(x_i)$.
	Adding and subtracting $W_i(x')d(x_i)$ and using
	\eqref{eq:weight-derivatives} gives
	\begin{equation*}
		\begin{aligned}
			\|G(x)-G(x')\|
			&\le
			d_0\|W(x)-W(x')\|
			+
			d_1\left(
			\sum_{i=1}^{m}W_i(x')^2|x_i-x_i'|^2
			\right)^{1/2}\\
			&\le (d_0C_W+d_1)\|x-x'\|.
		\end{aligned}
	\end{equation*}
	Set $C_1:=d_0C_W+d_1$.
	For any $y\in\mathbb R^m$, the product rule gives
	\begin{equation*}
		\nabla G(x)^\top y
		=
		\sum_{i=1}^{m}y_i d(x_i)\nabla W_i(x)
		+
		\sum_{i=1}^{m}y_iW_i(x)d'(x_i)e_i,
	\end{equation*}
	where $e_i$ is the $i$th coordinate vector.
	By \eqref{eq:weighted-first-bound} and
	$\sum_iW_i(x)^2\le1$,
	$\|\nabla G(x)^\top y\|\le C_1\|y\|_\infty$.
	
	To control the difference of Jacobians, write
	\begin{equation*}
		\begin{aligned}
			(\nabla G(x)-\nabla G(x'))^\top y
			={}&
			\sum_i y_i d(x_i)
			\bigl(\nabla W_i(x)-\nabla W_i(x')\bigr)+\sum_i y_i\bigl(d(x_i)-d(x_i')\bigr)\nabla W_i(x')\\
			&+\sum_i y_i\bigl(W_i(x)-W_i(x')\bigr)d'(x_i)e_i+\sum_i y_iW_i(x')
			\bigl(d'(x_i)-d'(x_i')\bigr)e_i.
		\end{aligned}
	\end{equation*}
	The four terms are bounded, respectively, by
	$d_0D_W$, $d_1C_W$, $d_1C_W$, and $d_2$, each multiplied by
	$\|y\|_\infty\|x-x'\|$.
	These bounds follow from
	\eqref{eq:weighted-second-bound},
	\eqref{eq:weighted-first-bound},
	\eqref{eq:weight-derivatives}, and the Lipschitz continuity of $d'$.
	With $C_2:=d_0D_W+2d_1C_W+d_2$, we therefore obtain
	\begin{equation*}
		\|(\nabla G(x)-\nabla G(x'))^\top y\|
		\le C_2\|y\|_\infty\|x-x'\|.
	\end{equation*}
	In particular,
	\begin{equation}
		\label{eq:G-derivatives}
		\|G(x)-G(x')\|\le C_1\|x-x'\|,
		\qquad
		\|\nabla G(x)-\nabla G(x')\|_{\rm op}
		\le C_2\|x-x'\|.
	\end{equation}
	
	Finally, $\sum_iW_i(x)^2\le1$ also implies
	$\|G(x)\|\le d_0$. Collecting the pointwise estimates gives
	\begin{equation}
		\label{eq:G-operator-bounds}
		\|G(x)\|\le d_0,
		\qquad
		\|\nabla G(x)\|_{\rm op}\le C_1,
		\qquad
		\|\nabla G(x)^\top y\|\le C_1\|y\|_\infty.
	\end{equation}
	
	\emph{Step 5: Smoothness of $Q_m$ and $\varphi_m$.}
	Since $Q_m(x)=\|G(x)\|^2$,
	$\nabla Q_m(x)=2\nabla G(x)^\top G(x)$.
	In particular,
	$\|\nabla Q_m(x)\|\le V_0:=2C_1d_0$.
	Moreover,
	\begin{equation*}
		\begin{aligned}
			\nabla Q_m(x)-\nabla Q_m(x')
			={}&
			2\nabla G(x)^\top\bigl(G(x)-G(x')\bigr)+
			2\bigl(\nabla G(x)-\nabla G(x')\bigr)^\top G(x').
		\end{aligned}
	\end{equation*}
	Equations~\eqref{eq:G-derivatives}--\eqref{eq:G-operator-bounds}
	therefore imply
	$\|\nabla Q_m(x)-\nabla Q_m(x')\|
	\le C_Q\|x-x'\|$,
	where $C_Q:=2(C_1^2+C_2d_0)$.
	Together with \eqref{eq:H-gradient-bound}, this proves (i) with
	$C_\varphi:=H_0+C_Q$, independently of $m$.
	
	\paragraph{(ii) Initial function-value gap.}
	Since $0\le S\le1$, $J\ge0$, and $Q_m\ge0$,
	$\varphi_m(x)\ge-4m$.
	At the origin, $S(0)=J(0)=0$,
	$W_1(0)=1$, and $W_i(0)=0$ for $i\ge2$.
	Thus $\varphi_m(0)=\mathcal K(0)=3/10$, and
	$\varphi_m(0)-\inf_x\varphi_m(x)\le4m+3/10\le5m$.
	
	\paragraph{(iii) Gradient barrier.}
	We show that $c_b:=1/100$ is admissible.
	For $i\in[m]$, define the tail term
	\begin{equation*}
		T_i(x):=
		\sum_{j=i+1}^{m}
		\left(\prod_{\ell=i+1}^{j-1}S(x_\ell)^2\right)
		(1-S(x_j))^2\mathcal K(x_j).
	\end{equation*}
	Using $(1-S)^2\le1-S^2$ and $\mathcal K\le k_+<1$,
	telescoping gives
	\begin{equation*}
		0\le T_i(x)
		\le k_+\left(1-\prod_{\ell=i+1}^{m}S(x_\ell)^2\right)
		\le k_+.
	\end{equation*}
	Direct differentiation yields
	\begin{equation*}
		\begin{aligned}
			\partial_i\varphi_m(x)
			={}&
			P_{i-1}(x)^2(1-S(x_i))^2\mathcal K'(x_i)-2P_{i-1}(x)^2(1-S(x_i))S'(x_i)\mathcal K(x_i)\\
			&+2P_{i-1}(x)^2S(x_i)S'(x_i)T_i(x)
			-4S'(x_i)+J'(x_i).
		\end{aligned}
	\end{equation*}
	Here $S'\ge0$, $J'\le0$, and
	$\mathcal K'=(3\gamma/5-2/5)\gamma'\le0$ because
	$\gamma'\ge0$ and $|\gamma|\le1/2$.
	Since $P_{i-1}\le1$ and $T_i\le k_+<1$, it follows that
	\begin{equation}
		\label{eq:barrier-unified-derivative}
		\partial_i\varphi_m(x)
		\le
		P_{i-1}(x)^2(1-S(x_i))^2\mathcal K'(x_i)
		-2S'(x_i)+J'(x_i)
		\le-2S'(x_i).
	\end{equation}
	
	Suppose, for a contradiction, that $x_m\le1/2$ and
	$\|\nabla\varphi_m(x)\|<c_b$.
	Let $k$ be the first index satisfying
	\begin{equation}
		\label{eq:barrier-frontier}
		S(x_k)\le\frac12.
	\end{equation}
	Such an index exists because $S(x_m)\le S(1/2)=1/4$.
	For every $i<k$, we have $S(x_i)>1/2$, so
	\eqref{eq:barrier-unified-derivative} implies
	\begin{equation}
		\label{eq:barrier-prefix-gradient}
		|\partial_i\varphi_m(x)|^2
		\ge4S'(x_i)^2
		=16q^2S(x_i)(1-S(x_i))
		\ge8q^2(1-S(x_i)).
	\end{equation}
	Summing over $i<k$ gives
	\begin{equation}
		\label{eq:barrier-prefix-deficit}
		\sum_{i<k}(1-S(x_i))
		\le\frac{\|\nabla\varphi_m(x)\|^2}{8q^2}
		<\frac{c_b^2}{8q^2}
		<\frac14.
	\end{equation}
	The elementary product inequality
	$\prod_i(1-a_i)\ge1-\sum_i a_i$ for $a_i\in[0,1]$
	therefore yields
	\begin{equation}
		\label{eq:barrier-prefix-product}
		P_{k-1}(x)
		\ge1-\sum_{i<k}(1-S(x_i))
		>\frac34.
	\end{equation}
	
	Since $S$ is strictly increasing on $(1/4,1)$ and
	$S(5/8)=1/2$, condition~\eqref{eq:barrier-frontier}
	implies $x_k\le5/8$.
	We distinguish two cases.
	
	If $x_k<-1$, then $J'(x_k)\le-1/2$.
	All other terms on the right-hand side of
	\eqref{eq:barrier-unified-derivative} are nonpositive, so
	$\partial_k\varphi_m(x)\le-1/2$.
	
	If $-1\le x_k\le5/8$, then $\gamma(x_k)=x_k/4$, and
	\begin{equation}
		\label{eq:barrier-profile-slope}
		\mathcal K'(x_k)
		=\frac{3x_k}{80}-\frac1{10}
		\le-\frac{49}{640}.
	\end{equation}
	Combining
	\eqref{eq:barrier-unified-derivative},
	\eqref{eq:barrier-frontier}, and
	\eqref{eq:barrier-prefix-product} gives
	\begin{equation*}
		\partial_k\varphi_m(x)
		\le
		P_{k-1}(x)^2(1-S(x_k))^2\mathcal K'(x_k)
		\le-\frac{441}{40960}
		<-\frac1{100}.
	\end{equation*}
	In either case,
	$\|\nabla\varphi_m(x)\|
	\ge|\partial_k\varphi_m(x)|>c_b$, a contradiction.
	This proves (iii).
	
	\paragraph{(iv) Zero-chain property.}
	Let $r\in\{0,\ldots,m-1\}$ and suppose that
	$\operatorname{supp}(x)\subseteq[r]$.
	We show that $\partial_j\varphi_m(x)=0$ for every $j\ge r+2$.
	
	First, $x_j=0$ and $S'(0)=J'(0)=0$, so
	$\partial_jH_m(x)=0$.
	For the terms in $Q_m$, distinguish two cases.
	If $i\le r+1$, then $W_i(x)^2\mathcal K(x_i)$ depends only on
	$x_1,\ldots,x_i$, and its derivative with respect to $x_j$ is zero.
	If $i\ge r+2$, the product $P_{i-1}(x)$ contains
	$S(x_{r+1})=S(0)=0$, so $W_i(x)=0$. Consequently,
	\begin{equation*}
		\partial_j\bigl(W_i(x)^2\mathcal K(x_i)\bigr)
		=
		2W_i(x)\partial_jW_i(x)\mathcal K(x_i)
		+
		W_i(x)^2\mathcal K'(x_i)\mathbf 1_{\{i=j\}}
		=0.
	\end{equation*}
	Thus $\partial_jQ_m(x)=0$ for every $j\ge r+2$, proving
	$\operatorname{supp}(\nabla\varphi_m(x))\subseteq[r+1]$.
	
	In addition, the potentially new coordinate satisfies
	$\partial_{r+1}\varphi_m(x)
	=P_r(x)^2\mathcal K'(0)=-P_r(x)^2/10$.
	Hence it is revealed precisely when $P_r(x)>0$.
\end{proof}

\section{Proof of Lemma~\ref{lem:new-C-instance}}
\label{app:ncc}
\label{app:ncc:construction}

We use the constants and derivative estimates established in
Appendix~\ref{app:common}. All constants introduced below are
independent of the chain length $m$.

\begin{proof}
	Write $u:=x/\beta$, $v:=\zeta y$, and
	$T_i(u):=W_i(u)\gamma(u_i)$. Define
	\[
	\mathcal F_m(u,v)
	:=
	H_m(u)-\frac12\sum_{i=1}^{m}(v^i)^\top B_2v^i
	+\frac1{\sqrt2}\sum_{i=1}^{m}
	\bigl(W_i(u)v_1^i-T_i(u)v_2^i\bigr).
	\]
	Then $F_{\rm C}(x,y)=\alpha\mathcal F_m(x/\beta,\zeta y)$,
	and $\zeta Y$ is the ball of radius $2$.
	
	\paragraph{Auxiliary derivative estimates.}
	Set $\gamma_0:=1/2$, $\gamma_1:=1/4$, and $\gamma_2:=1/8$.
	The definition of $\gamma$ gives
	$\|\gamma\|_\infty\le\gamma_0$,
	$\|\gamma'\|_\infty\le\gamma_1$, and
	$\operatorname{Lip}(\gamma')\le\gamma_2$.
	For any $a\in\mathbb R^m$, the product rule yields
	\begin{equation}
		\label{eq:new-C-T-gradient}
		\sum_i a_i\nabla T_i(u)
		=
		\sum_i a_i\gamma(u_i)\nabla W_i(u)
		+\sum_i a_iW_i(u)\gamma'(u_i)e_i.
	\end{equation}
	By \eqref{eq:weighted-first-bound} and
	$\sum_iW_i(u)^2\le1$,
	\begin{equation}
		\label{eq:new-C-T-first-bound}
		\left\|\sum_i a_i\nabla T_i(u)\right\|
		\le C_T\|a\|_\infty,
		\qquad C_T:=\gamma_0C_W+\gamma_1.
	\end{equation}
	Subtracting \eqref{eq:new-C-T-gradient} at $u$ and $u'$ gives
	\begin{equation}
		\label{eq:new-C-T-difference}
		\begin{aligned}
			\sum_i a_i\bigl(\nabla T_i(u)-\nabla T_i(u')\bigr)
			={}&
			\sum_i a_i\gamma(u_i)
			\bigl(\nabla W_i(u)-\nabla W_i(u')\bigr)\\
			&+\sum_i a_i\bigl(\gamma(u_i)-\gamma(u_i')\bigr)\nabla W_i(u')\\
			&+\sum_i a_i\bigl(W_i(u)-W_i(u')\bigr)\gamma'(u_i)e_i\\
			&+\sum_i a_iW_i(u')
			\bigl(\gamma'(u_i)-\gamma'(u_i')\bigr)e_i.
		\end{aligned}
	\end{equation}
	By \eqref{eq:weighted-second-bound},
	\eqref{eq:weighted-first-bound},
	\eqref{eq:weight-derivatives}, and the bounds on $\gamma$,
	the four terms have norms at most
	$\gamma_0D_W$, $\gamma_1C_W$, $\gamma_1C_W$, and $\gamma_2$,
	respectively, multiplied by $\|a\|_\infty\|u-u'\|$.
	Thus, with $D_T:=\gamma_0D_W+2\gamma_1C_W+\gamma_2$,
	\begin{equation}
		\label{eq:new-C-T-second-bound}
		\left\|\sum_i a_i\bigl(\nabla T_i(u)-\nabla T_i(u')\bigr)\right\|
		\le D_T\|a\|_\infty\|u-u'\|.
	\end{equation}
	The weighted first-derivative bounds also imply that $W$ and
	$T:=(T_1,\ldots,T_m)$ are $C_W$- and $C_T$-Lipschitz,
	respectively, in the Euclidean norm.
	
	\paragraph{(i) Smoothness, dual concavity, and the value function.}
	\emph{Joint smoothness.}
	Direct differentiation gives
	\[
	\begin{aligned}
		\nabla_u\mathcal F_m(u,v)
		&=\nabla H_m(u)+\frac1{\sqrt2}
		\sum_i\bigl(v_1^i\nabla W_i(u)-v_2^i\nabla T_i(u)\bigr),\\
		[\nabla_v\mathcal F_m(u,v)]_{v^i}
		&=-B_2v^i+\frac1{\sqrt2}\bigl(W_i(u),-T_i(u)\bigr)^\top.
	\end{aligned}
	\]
	Set $A:=H_0+\sqrt2(D_W+D_T)$ and
	$B:=\sqrt{C_W^2+C_T^2}/\sqrt2$.
	For $\|v\|,\|v'\|\le2$, the preceding derivative estimates give
	$\|\nabla_u\mathcal F_m(u,v)-\nabla_u\mathcal F_m(u',v')\|
	\le A\|u-u'\|+B\|v-v'\|$ and
	$\|\nabla_v\mathcal F_m(u,v)-\nabla_v\mathcal F_m(u',v')\|
	\le B\|u-u'\|+(5/2)\|v-v'\|$.
	Here we used $\|B_2\|_{\rm op}=5/2$ and
	$\|(v_1^i)_i\|_\infty,\|(v_2^i)_i\|_\infty\le2$.
	Define the dimension-independent constant
	\[
	C_*:=\max\left\{
	1,C_\varphi,
	\left\|\begin{pmatrix}A&B\\B&5/2\end{pmatrix}\right\|_{\rm op}
	\right\}.
	\]
	Then $\nabla\mathcal F_m$ is $C_*$-Lipschitz on
	$\mathbb R^m\times\zeta Y$, so the chain rule gives
	$\operatorname{Lip}(\nabla F_{\rm C})
	\le\alpha C_*\max\{\beta^{-2},\zeta^2\}$.
	
	Since $\beta\zeta=96C_*\varepsilon/(c_bLD_Y)$ and
	$\alpha=72C_*\varepsilon^2/(c_b^2L)$, a sufficiently small
	universal $c_0>0$ ensures $\beta\zeta\le1$ and
	$\Delta/(10\alpha)\ge2$ throughout the stated accuracy range.
	In particular, $m\ge1$. Using $\alpha=L\beta^2/(8C_*)$,
	we obtain
	\begin{equation}
		\label{eq:new-C-app-smoothness}
		\operatorname{Lip}(\nabla F_{\rm C})
		\le\frac{\alpha C_*}{\beta^2}=\frac L8.
	\end{equation}
	
	\emph{Dual strong concavity.}
	The eigenvalues of $B_2$ are $1/2$ and $5/2$.
	Since $\nabla_{yy}^2F_{\rm C}
	=-\alpha\zeta^2\operatorname{diag}(B_2,\ldots,B_2)$,
	the function $F_{\rm C}(x,\cdot)$ is
	$\alpha\zeta^2/2$-strongly concave.
	
	\emph{Value function and initial gap.}
	For fixed $u$, the unconstrained dual maximizer is
	\begin{equation}
		\label{eq:new-C-dual-maximizer}
		v^{i,*}(u)
		=\frac{W_i(u)}{\sqrt2}B_2^{-1}
		\bigl(1,-\gamma(u_i)\bigr)^\top,
		\qquad i\in[m].
	\end{equation}
	Using $\|B_2^{-1}\|_{\rm op}=2$, $|\gamma|\le1/2$, and
	$\sum_iW_i(u)^2\le1$, we obtain
	$\|v^*(u)\|^2\le2\sum_iW_i(u)^2(1+\gamma(u_i)^2)\le5/2<4$.
	Thus $v^*(u)\in\zeta Y$.
	Since $B_2^{-1}=\frac15\left(\begin{smallmatrix}6&4\\4&6\end{smallmatrix}\right)$,
	substituting \eqref{eq:new-C-dual-maximizer} shows that the
	maximum contributed by block $i$ is
	$\frac14W_i(u)^2(1,-\gamma(u_i))B_2^{-1}
	(1,-\gamma(u_i))^\top=W_i(u)^2\mathcal K(u_i)$.
	Consequently,
	\begin{equation}
		\label{eq:new-C-app-value}
		\Phi_{\rm C}(x)=\max_{y\in Y}F_{\rm C}(x,y)
		=\alpha\varphi_m(x/\beta).
	\end{equation}
	Lemma~\ref{lem:common-chain}(ii) and the choice of $m$ give
	$\Phi_{\rm C}(0)-\inf_x\Phi_{\rm C}(x)
	\le5\alpha m\le\Delta/2\le\Delta$.
	Together with the domain choices, this proves
	$F_{\rm C}\in\mathcal F_{\rm NC\text{-}C}(L,\Delta,D_Y)$
	and completes (i).
	
	\paragraph{(ii) Unbiasedness, variance, and mean-square smoothness.}
	Let $\mathsf R$ be the orthogonal map that exchanges the two
	coordinates in each dual block.
	Comparing \eqref{eq:new-C-gradient} with $\nabla F_{\rm C}$ gives
	\begin{equation}
		\label{eq:new-C-app-oracle-decomposition}
		G_{\rm C}(z;\omega)
		=\nabla F_{\rm C}(z)
		+\left(\frac{\omega}{p}-1\right)\mathcal A(z),
		\qquad
		\mathcal A(z):=(0,\alpha\zeta\mathsf Rv).
	\end{equation}
	The Bernoulli identities give
	$\mathbb E_\omega(\omega/p-1)=0$ and
	$\mathbb E_\omega(\omega/p-1)^2
	=(1-p)/p=\sigma^2/(4\alpha^2\zeta^2)$, including when
	$\sigma=0$. Hence the oracle is unbiased.
	Since $\|\mathcal A(z)\|\le2\alpha\zeta$, its variance satisfies
	$\mathbb E_\omega\|G_{\rm C}(z;\omega)-\nabla F_{\rm C}(z)\|^2
	=((1-p)/p)\|\mathcal A(z)\|^2\le\sigma^2$.
	
	For the MSS bound, use the same sample at $z=(x,y)$ and
	$z'=(x',y')$.
	Orthogonality of $\mathsf R$ gives
	$\|\mathcal A(z)-\mathcal A(z')\|
	=\alpha\zeta^2\|y-y'\|$.
	The centered coefficient in
	\eqref{eq:new-C-app-oracle-decomposition} makes the cross term
	vanish in expectation. Thus the expected squared gradient
	difference equals
	$\|\nabla F_{\rm C}(z)-\nabla F_{\rm C}(z')\|^2
	+((1-p)/p)\|\mathcal A(z)-\mathcal A(z')\|^2$.
	By \eqref{eq:new-C-app-smoothness}, this is at most
	$(L^2/64+4\sigma^2/D_Y^2)\|z-z'\|^2$.
	Since $\bar L\ge L$ and $\sigma\le D_Y\bar L/4$,
	the coefficient satisfies
	$L^2/64+4\sigma^2/D_Y^2\le17\bar L^2/64\le\bar L^2$.
	This proves (ii).
	
	\paragraph{(iii) Moreau-envelope stationarity barrier.}
	By \eqref{eq:new-C-app-value} and $C_\varphi\le C_*$,
	$\operatorname{Lip}(\nabla\Phi_{\rm C})
	\le\alpha C_\varphi/\beta^2\le L/8$.
	Set $b:=3\varepsilon/2$.
	The parameter choices give $\alpha c_b/(2\beta)=b$ and
	$b/(L\beta)=c_b/(16C_*)\le1$, where we use
	$c_b=1/100$ and $C_*\ge1$.
	Thus $b\le\min\{L\beta,\alpha c_b/(2\beta)\}$.
	Since $X=\mathbb R^m$, Lemma~\ref{nccmss:lem:moreau-transfer}
	yields $S_{\rm OS}(x)>3\varepsilon/2$ for every $x$ with
	$x_m=0$. This proves (iii).
\end{proof}

\section{Proof of Lemma~\ref{lem:bv-class}}
\label{app:bv}
\label{app:bv:construction}

For $t\ge0$, define the clipping function
\[
\chi(t):=
\begin{cases}
	t, & 0\le t\le1/4,\\
	\frac43t^3-3t^2+\frac94t-\frac7{48}, & 1/4<t<3/4,\\
	5/12, & t\ge3/4,
\end{cases}
\]
and extend it to $\mathbb R$ by odd symmetry.
Using the constants from Appendix~\ref{app:common}, set
\begin{equation}
	\label{eq:bv-constants}
	C_{\rm B}:=\max\{1,8(H_0+2\sqrt2C_2),128C_1^2,2C_\varphi\},
	\qquad C_{\rm V}:=2(C_1^2+k_+).
\end{equation}
Both constants are independent of $m$.

\begin{proof}
	Write $u:=x/\beta$ and $\rho:=\sqrt{2\mu\alpha}$.
	Choose $c_0\le c_b/\sqrt{10240C_{\rm B}}$.
	Then $\alpha=256C_{\rm B}\varepsilon^2/(c_b^2L)\le\Delta/40$,
	so $m=\lfloor\Delta/(5\alpha)\rfloor\ge1$.
	
	\paragraph{(i)--(ii) Admissibility and the value function.}
	\emph{Dual maximization and initial gap.}
	For fixed $x$, the unconstrained dual maximizer is
	$y^*(x)=(\rho/\mu)G(u)$.
	By \eqref{eq:G-operator-bounds},
	$\|G(u)\|\le\sqrt{k_+}<\sqrt2$, so
	$\|y^*(x)\|<2\sqrt{\alpha/\mu}$.
	Thus $y^*(x)\in Y$, and substitution gives
	$\Phi_{\rm BV}(x)=\alpha H_m(u)+\alpha\|G(u)\|^2
	=\alpha\varphi_m(u)$, proving \eqref{eq:bv-value}.
	Lemma~\ref{lem:common-chain}(ii) then yields
	$\Phi_{\rm BV}(0)-\inf_x\Phi_{\rm BV}(x)\le5\alpha m\le\Delta$.
	Moreover, $\nabla_{yy}^2F_{\rm BV}=-\mu I$, so
	$F_{\rm BV}(x,\cdot)$ is $\mu$-strongly concave.
	
	\emph{Joint smoothness.}
	Direct differentiation gives
	\begin{align}
		\nabla_xF_{\rm BV}(x,y)
		&=\frac{\alpha}{\beta}\nabla H_m(u)
		+\frac{\rho}{\beta}\nabla G(u)^\top y,
		\label{eq:bv-mean-x-gradient}\\
		\nabla_yF_{\rm BV}(x,y)
		&=\rho G(u)-\mu y.
		\label{eq:bv-mean-y-gradient}
	\end{align}
	Let $z=(x,y)$ and $z'=(x',y')$ belong to $X\times Y$.
	Using $\rho\|y\|\le2\sqrt2\,\alpha$ together with
	\eqref{eq:H-gradient-bound}, \eqref{eq:G-derivatives}, and
	\eqref{eq:G-operator-bounds}, we obtain
	\begin{equation}
		\label{eq:bv-primal-gradient-difference}
		\|\nabla_xF_{\rm BV}(z)-\nabla_xF_{\rm BV}(z')\|
		\le
		\frac{\alpha(H_0+2\sqrt2C_2)}{\beta^2}\|x-x'\|
		+\frac{\rho C_1}{\beta}\|y-y'\|.
	\end{equation}
	Similarly,
	\begin{equation}
		\label{eq:bv-dual-gradient-difference}
		\|\nabla_yF_{\rm BV}(z)-\nabla_yF_{\rm BV}(z')\|
		\le\frac{\rho C_1}{\beta}\|x-x'\|+\mu\|y-y'\|.
	\end{equation}
	Since $\alpha=L\beta^2/C_{\rm B}$ and $\kappa\ge2$,
	\eqref{eq:bv-constants} gives
	$\alpha(H_0+2\sqrt2C_2)/\beta^2\le L/8$,
	$\rho C_1/\beta=LC_1\sqrt{2/(\kappa C_{\rm B})}\le L/8$,
	and $\mu\le L/2$.
	Combining the two block estimates and applying Cauchy--Schwarz
	therefore yields
	$\|\nabla F_{\rm BV}(z)-\nabla F_{\rm BV}(z')\|
	\le(\sqrt{29}L/8)\|z-z'\|\le L\|z-z'\|$.
	Together with the domain choices, dual strong concavity, and
	the initial-gap bound, this proves
	$F_{\rm BV}\in\mathcal F_{\rm NC\text{-}SC}(L,\mu,\Delta)$.
	
	\paragraph{(iii) Unbiasedness and bounded variance.}
	On the middle branch,
	$\chi'(t)=4(t-3/4)^2$ and $\chi''(t)=8t-6$.
	The values and first derivatives match at both endpoints.
	The odd extension is linear near zero, and its piecewise second
	derivative has absolute value at most $4$.
	Hence $\chi\in C^{1,1}(\mathbb R)$, with
	$|\chi(t)|\le5/12$, $0\le\chi'(t)\le1$,
	$\chi(0)=0$, and $\chi'(0)=1$.
	Consequently, $\chi_\beta(t):=\beta\chi(t/\beta)$ satisfies
	$\|\chi_\beta(y)\|_\infty\le5\beta/12$,
	$|\chi_\beta'(t)|\le1$, $\chi_\beta(0)=0$, and
	$\chi_\beta'(0)=1$.
	
	Write
	$G_{\rm BV}(z;\omega)=\nabla F_{\rm BV}(z)
	+(\omega/p-1)\mathcal A_{\rm BV}(z)$, where
	\[
	\mathcal A_{\rm BV}(z):=
	\rho
	\begin{pmatrix}
		\beta^{-1}\nabla G(u)^\top\chi_\beta(y)\\
		\bigl(G_i(u)\chi_\beta'(y_i)\bigr)_{i=1}^{m}
	\end{pmatrix}.
	\]
	By \eqref{eq:G-operator-bounds}, the two blocks inside this
	vector have norms at most $5C_1/12$ and $\sqrt{k_+}$,
	respectively. Thus
	$\|\mathcal A_{\rm BV}(z)\|^2
	\le2\mu\alpha((5C_1/12)^2+k_+)\le C_{\rm V}\mu\alpha$.
	The choice of $p$ gives $0<p\le1$,
	$\mathbb E_\omega(\omega/p-1)=0$, and
	$\mathbb E_\omega(\omega/p-1)^2
	=(1-p)/p=\sigma^2/(C_{\rm V}\mu\alpha)$.
	These identities remain valid when $\sigma=0$, in which case
	$p=1$.
	Hence the oracle is unbiased and
	$\mathbb E_\omega\|G_{\rm BV}(z;\omega)-\nabla F_{\rm BV}(z)\|^2
	=((1-p)/p)\|\mathcal A_{\rm BV}(z)\|^2\le\sigma^2$.
	
	\paragraph{(iv) Coordinate revelation.}
	Set $\eta:=\omega/p-1$. The oracle blocks are
	\[
	\begin{aligned}
		G_{{\rm BV},x}(z;\omega)
		&=\frac{\alpha}{\beta}\nabla H_m(u)
		+\frac{\rho}{\beta}\sum_i
		\bigl(y_i+\eta\chi_\beta(y_i)\bigr)\nabla G_i(u),\\
		[G_{{\rm BV},y}(z;\omega)]_i
		&=\rho G_i(u)-\mu y_i+\eta\rho G_i(u)\chi_\beta'(y_i).
	\end{aligned}
	\]
	Suppose a query is supported on the first $j$ coordinate pairs,
	where $0\le j<m$.
	The separability of $H_m$ and $S'(0)=J'(0)=0$ imply
	$\operatorname{supp}(\nabla H_m(u))\subseteq[j]$.
	For $i\le j$, $G_i$ depends only on $u_1,\ldots,u_i$;
	for $i>j$, its coefficient in the primal response is
	$y_i+\eta\chi_\beta(y_i)=0$.
	Thus the primal response is supported on $[j]$.
	
	For $i>j$, the identities $y_i=0$ and $\chi_\beta'(0)=1$
	give $[G_{{\rm BV},y}(z;\omega)]_i=(\rho\omega/p)G_i(u)$.
	If $i\ge j+2$, the product defining $W_i(u)$ contains
	$S(u_{j+1})=S(0)=0$, so $G_i(u)=0$.
	The only potentially nonzero response component beyond the
	first $j$ pairs is therefore the dual component at $j+1$,
	which vanishes when $\omega=0$.
	A successful round can reveal at most the next pair, whereas
	a failed round reveals no pair beyond the existing prefix.
	The case $j=m$ is immediate.
	
	Let $j_t$ be the largest index of a pair containing a coordinate
	revealed by the end of round $t$, with $j_0:=0$.
	By zero-respectingness, every query in round $t$ is supported on
	the first $j_{t-1}$ pairs.
	Since all queries are chosen before that round's responses and
	share $\omega_t$, the preceding support calculation applies
	to the entire batch.
	Hence $j_t\le\min\{m,j_{t-1}+\omega_t\}$, and induction gives
	$j_R\le\sum_{t=1}^{R}\omega_t$.
	If $\widehat x_m\ne0$, the output restriction requires that
	$x_m$ has been revealed, so $j_R=m$. Therefore,
	\begin{equation}
		\label{eq:bv-delayed-revelation}
		\widehat x_m=0
		\qquad\text{whenever}\qquad
		\sum_{t=1}^{R}\omega_t<m.
	\end{equation}
	Exact function values do not enlarge the revealed-coordinate
	set and therefore do not affect this argument.
	This proves (iv) for arbitrary finite nonadaptive same-sample
	batches and completes the proof.
\end{proof}

\section{Proof of Lemma~\ref{lem:mss-class}}
\label{app:mss:construction}

\begin{proof}
	Recall that $h(y):=\delta\lambda\varphi_m(y/\lambda)$,
	$F_{\rm MS}(x,y)=h(y)-\mu\|y-ax\|^2$, and
	$\ell:=\delta C_\varphi/\lambda$.
	By Lemma~\ref{lem:common-chain}(i), $h$ is $\ell$-smooth.
	Throughout, we use $\kappa\ge8$, $a^2=\kappa/8$, and
	$\ell\le\mu=L/\kappa$.
	
	\paragraph{(i) Membership in the NC--SC class.}
	\emph{Dual strong concavity and existence of the maximizer.}
	The smoothness inequality for $h$ implies that
	$F_{\rm MS}(x,\cdot)$ is $(2\mu-\ell)$-strongly concave,
	and hence $\mu$-strongly concave.
	Its quadratic upper bound at $y=0$ tends to $-\infty$
	as $\|y\|\to\infty$.
	Continuity and strong concavity therefore guarantee a unique
	dual maximizer $y^*(x)$.
	
	\emph{Joint smoothness.}
	The quadratic term $-\mu\|y-ax\|^2$ has gradient Lipschitz
	constant $2\mu(1+a^2)=2L/\kappa+L/4\le L/2$.
	Since $\ell\le\mu\le L/8$, we obtain
	\begin{equation}
		\label{eq:mss-joint-smoothness}
		\|\nabla F_{\rm MS}(z)-\nabla F_{\rm MS}(z')\|
		\le(\ell+L/2)\|z-z'\|\le(5L/8)\|z-z'\|.
	\end{equation}
	In particular, $F_{\rm MS}$ is jointly $L$-smooth.
	
	\emph{Initial value-function gap.}
	Let $h_*:=\inf_{y\in\mathbb R^m}h(y)$.
	Lemma~\ref{lem:common-chain}(ii) and the choice of $m$ give
	$h(0)-h_*\le5\delta\lambda m\le\Delta/2$.
	Applying the smoothness inequality at $-\nabla h(0)/\ell$ yields
	\begin{equation}
		\label{eq:mss-h-gradient-at-zero}
		\|\nabla h(0)\|^2\le2\ell\bigl(h(0)-h_*\bigr).
	\end{equation}
	For every $x$, choosing $y=ax$ gives
	$\Phi_{\rm MS}(x)\ge h(ax)\ge h_*$.
	At $x=0$, the smoothness inequality gives the quadratic upper
	bound $F_{\rm MS}(0,y)\le h(0)+\langle\nabla h(0),y\rangle
	-(2\mu-\ell)\|y\|^2/2$.
	Maximizing this bound yields
	\begin{equation}
		\label{eq:mss-value-at-zero}
		\Phi_{\rm MS}(0)\le h(0)+\frac{\|\nabla h(0)\|^2}{2(2\mu-\ell)}.
	\end{equation}
	Combining \eqref{eq:mss-h-gradient-at-zero} and
	\eqref{eq:mss-value-at-zero} with $\inf_x\Phi_{\rm MS}(x)\ge h_*$
	gives
	$\Phi_{\rm MS}(0)-\inf_x\Phi_{\rm MS}(x)
	\le[2\mu/(2\mu-\ell)](h(0)-h_*)\le\Delta$.
	This proves
	$F_{\rm MS}\in\mathcal F_{\rm NC\text{-}SC}(L,\mu,\Delta)$.
	
	\paragraph{(ii) Unbiasedness, variance, and mean-square smoothness.}
	Define $\mathcal A(z):=(0,\delta\nabla Q_m(y/\lambda))$.
	Then $G_{\rm MS}(z;\omega)=\nabla F_{\rm MS}(z)
	+(\omega/p-1)\mathcal A(z)$.
	The parameter choices ensure $0<p\le1$.
	The Bernoulli identities
	$\mathbb E_\omega(\omega/p-1)=0$ and
	$\mathbb E_\omega(\omega/p-1)^2=(1-p)/p$
	prove unbiasedness and, using $\|\nabla Q_m\|\le V_0$, give
	$\mathbb E_\omega\|G_{\rm MS}(z;\omega)-\nabla F_{\rm MS}(z)\|^2
	\le((1-p)/p)\delta^2V_0^2=(1-p)\sigma^2\le\sigma^2$.
	
	For the MSS bound, use the same sample at $z$ and $z'$.
	The centered coefficient makes the cross term vanish, so the
	expected squared gradient difference equals
	$\|\nabla F_{\rm MS}(z)-\nabla F_{\rm MS}(z')\|^2
	+((1-p)/p)\|\mathcal A(z)-\mathcal A(z')\|^2$.
	Since $Q_m$ is $C_Q$-smooth,
	$\|\mathcal A(z)-\mathcal A(z')\|
	\le(\delta C_Q/\lambda)\|y-y'\|$.
	Moreover, \eqref{eq:mss-scales-en} gives
	$\delta C_Q/(\lambda\sqrt p)
	=\bar L C_Q/[4(C_\varphi+C_Q)]\le\bar L/4$.
	Together with \eqref{eq:mss-joint-smoothness} and $1-p\le1$,
	these estimates bound the expected squared gradient difference
	by $(25L^2/64)\|z-z'\|^2+(\bar L^2/16)\|y-y'\|^2$.
	Since $L\le\bar L$, this is at most
	$(29\bar L^2/64)\|z-z'\|^2\le\bar L^2\|z-z'\|^2$.
	This proves (ii).
	
	\paragraph{(iii) Coordinate revelation.}
	Using $\varphi_m=H_m+Q_m$, the oracle blocks are
	\[
	\begin{aligned}
		G_{{\rm MS},x}(x,y;\omega)
		&=2\mu a(y-ax),\\
		G_{{\rm MS},y}(x,y;\omega)
		&=\delta\nabla H_m(y/\lambda)
		+\frac{\delta\omega}{p}\nabla Q_m(y/\lambda)-2\mu(y-ax).
	\end{aligned}
	\]
	Suppose a query is supported on the first $r$ coordinate pairs,
	where $0\le r<m$.
	The separability of $H_m$ and $S'(0)=J'(0)=0$ give
	$\operatorname{supp}(\nabla H_m(y/\lambda))\subseteq[r]$.
	The zero-chain property of $\varphi_m$ and
	$Q_m=\varphi_m-H_m$ then imply
	$\operatorname{supp}(\nabla Q_m(y/\lambda))\subseteq[r+1]$.
	Thus the response is supported on the first $r+1$ pairs,
	and its only potentially nonzero component beyond the first
	$r$ pairs is
	$[G_{{\rm MS},y}(x,y;\omega)]_{r+1}
	=(\delta\omega/p)[\nabla Q_m(y/\lambda)]_{r+1}$.
	A successful round can therefore reveal at most the next pair,
	whereas a failed round reveals no pair beyond the existing prefix.
	The case $r=m$ is immediate.
	
	Let $j_t$ be the largest index of a pair containing a coordinate
	revealed by the end of round $t$, with $j_0:=0$.
	By zero-respectingness, every query in round $t$ is supported
	on the first $j_{t-1}$ pairs.
	All queries are chosen before that round's responses and share
	$\omega_t$, so the support calculation applies to the entire batch.
	Hence $j_t\le\min\{m,j_{t-1}+\omega_t\}$, and induction gives
	$j_N\le\sum_{t=1}^{N}\omega_t$.
	If $\widehat x_m\ne0$, the output restriction requires that
	$x_m$ has been revealed, so $j_N=m$.
	Consequently, $\widehat x_m=0$ whenever
	$\sum_{t=1}^{N}\omega_t<m$.
	Exact function values do not enlarge the revealed-coordinate
	set and therefore do not affect this argument.
	This proves (iii) for arbitrary finite nonadaptive same-sample
	batches.
	
	\paragraph{(iv) Transfer of the gradient barrier.}
	The first-order condition is
	$\nabla h(y^*(x))=2\mu(y^*(x)-ax)$.
	Subtracting this identity at $x$ and $x'$ and using
	$\ell$-smoothness gives
	$(2\mu-\ell)\|y^*(x)-y^*(x')\|\le2\mu a\|x-x'\|$.
	Thus the maximizer map is Lipschitz continuous, and the
	value-function gradient formula gives
	$\nabla\Phi_{\rm MS}(x)=2\mu a(y^*(x)-ax)$.
	The first-order condition and $\ell$-smoothness imply
	$\|\nabla h(ax)\|\le(2\mu+\ell)\|y^*(x)-ax\|$.
	Substituting the value-function gradient formula proves
	\eqref{eq:mss-transfer-upper}.
	
	For any $x$ with $x_m=0$, Lemma~\ref{lem:common-chain}(iii)
	gives $\|\nabla h(ax)\|
	=\delta\|\nabla\varphi_m(ax/\lambda)\|\ge\delta c_b$.
	By \eqref{eq:mss-transfer-upper}, $\ell\le\mu$, and
	$a\delta c_b=4\varepsilon$, we conclude that
	$\|\nabla\Phi_{\rm MS}(x)\|
	\ge a\delta c_b/(1+\ell/(2\mu))
	\ge2a\delta c_b/3=8\varepsilon/3$.
	This proves (iv) and completes the proof.
\end{proof}

\end{document}